\documentclass[11pt,reqno]{amsart}
\usepackage{amsmath}

\usepackage{float}
\usepackage{times}
\usepackage[T1]{fontenc}
\usepackage{mathrsfs}
\usepackage{latexsym}
\usepackage[dvips]{graphics}
\usepackage{graphicx}
\usepackage{fancyhdr,amssymb}
\usepackage{epsfig}
\usepackage{systeme}
\usepackage{amsmath,amsfonts,amsthm,amssymb,amscd}
\input amssym.def
\input amssym.tex
\usepackage{color}
\usepackage{hyperref}
\usepackage{url}
\usepackage{breakurl}
\usepackage{tikz}
\newcommand{\bburl}[1]{\textcolor{blue}{\url{#1}}}

\usepackage{comment}
\newcommand{\js}[1]{{#1\overwithdelims () p}}
\newcommand{\zsum}[1]{ \sum_{#1 = 0}^{p-1} }
\newcommand{\osum}[1]{ \sum_{#1 = 1}^{p-1} }

\everymath{\displaystyle}

\usepackage[margin=1in]{geometry}
\usepackage{mathtools}
\usepackage{parskip}[=v1]
\usepackage{indentfirst}
\usepackage{listings}
\newtheorem{theorem}{Theorem}[section]

\newtheorem{lemma}[theorem]{Lemma}
\newtheorem*{theorem*}{Theorem}
\newtheorem{conjecture}[theorem]{Conjecture}
\newtheorem{remark}[theorem]{Remark}

\newtheorem{definition}[theorem]{Definition}

\newcommand\be{\begin{equation}}
\newcommand\ee{\end{equation}}
\newcommand\bea{\begin{eqnarray}}
\newcommand\eea{\end{eqnarray}}

\newcommand{\Z}{\ensuremath{\mathbb{Z}}}
\newcommand{\Q}{\mathbb{Q}}

\newcommand{\twocase}[5]{#1 \begin{cases} #2 & \text{{\rm #3}}\\ #4
&\text{{\rm #5}} \end{cases}   }

\newcommand{\threecase}[7]{#1 \begin{cases} #2 & \text{{\rm #3}}\\ #4
&\text{{\rm #5}}\\ #6 & \text{{\rm #7}} \end{cases}   }
\newcommand{\ncr}[2]{{#1 \choose #2}}

\title{Biases in Moments of Dirichlet Coefficients of Elliptic Curve Families}
\date{\today}

\author{Yan (Roger) Weng}
\email{\textcolor{blue}{\href{mailto:Roger\_weng2018@outlook.com}{Roger\_weng2018@outlook.com}}}
\address{AFFILIATION}

\date{\today}

\begin{document}

  \begin{titlepage}
    \centering
    \vspace*{0.6in}
    \bgroup
    \Huge\bfseries Biases in Moments of Dirichlet Coefficients of Elliptic Curve Families\par
    \egroup
    \vspace{0.5in}
    \bgroup
    \Large\bfseries Yan (Roger) Weng\\[0.1in]
    \egroup \large
    Peddie School\\ Hightstown, NJ, USA \par
    \vspace{0.3in}
    \large under the direction of\\
    \vspace{0.3in}
    \bgroup
    \Large\bfseries Steven. J. Miller\\[0.1in]
    \egroup
    \large Williams College\\ Williamstown, MA, USA\par
   \vspace{2in}
   \today
  \end{titlepage}

\begin{abstract} Elliptic curves arise in many important areas of modern number theory. One way to study them is take local data, the number of solutions modulo $p$, and create an $L$-function. The behavior of this global object is related to two of the seven Clay Millenial Problems: the Birch and Swinnerton-Dyer Conjecture and the Generalized Riemann Hypothesis. We study one-parameter families over $\Q(T)$, which are of the form $y^2=x^3+A(T)x+B(T)$, with non-constant $j$-invariant. We define the $r$\textsuperscript{th} moment of an elliptic curve to be $A_{r,E}(p) :=  \frac1{p} \sum_{t \bmod p} a_t(p)^r$, where $a_t(p)$ is $p$ minus the number of solutions to $y^2 = x^3 + A(t)x + B(t) \bmod p$. Rosen and Silverman showed biases in the first moment equal the rank of the Mordell-Weil group of rational solutions.

Michel proved that $pA_{2,E}(p)=p^2+O(p^{3/2})$. Based on several special families where computations can be done in closed form, Miller in his thesis conjectured that the largest lower-order term in the second moment that does not average to $0$ is on average negative. He further showed that such a negative bias has implications in the distribution of zeros of the elliptic curve $L$-function near the central point. To date, evidence for this conjecture is limited to special families. In addition to studying some additional families where the calculations can be done in closed form, we also systematically investigate families of various rank. These are the first general tests of the conjecture; while we cannot in general obtain closed form solutions, we discuss computations which support or contradict the conjecture. We then generalize to higher moments, and see evidence that the bias continues in the even moments.
\end{abstract}

\maketitle

\tableofcontents

\section{Introduction}

Elliptic curves play a major role in many problems modern number theory for a variety of reasons. First, they are part of a tradition going back thousands of years: counting how many rational points satisfy a polynomial. The best known example of this are the Pythagorean triples: finding integer sides $a, b$ and $c$ of a right triangle require $a^2 + b^2 = c^2$, which is equivalent to finding rational points on the circle $x^2 + y^2 = 1$. A more recent application was the discovery that elliptic curves could be used to create strong crypto-systems.

In this paper we explore some properties related to elliptic curves, in particular to biases in the distribution of the Dirichlet coefficients of the associated $L$-functions. To put our work in perspective, we go back a little over a hundred years to the International Congress of Mathematicians in Paris in 1900. David Hilbert \cite{BY} presented 10 of 23 problems (many of which are still open) that he believed would be great guides for mathematics in the coming years. To start this century, the Clay Mathematics Institute proposed 7 problems \cite{MP}. One has since been solved (the Poincare Conjecture), but the other 6 are open; our work intersects two of these: the Generalized Riemann Hypothesis, and the Birch and Swinnerton-Dyer Conjecture.

We briefly summarize our problem, expanding into greater detail below. An elliptic curve is of the form $y^2 = x^3 + ax + b$ with $a, b\in \Z$, and more generally we can consider one-parameter families $y^2 = x^3 + a(T) x + b(T)$ with $a(T), b(T) \in \Z[T]$. If we specialize $T$ to an integer $t$ we get an elliptic curve $E_t$, and for each prime $p$ we can count how many $(x,y)$ solve $y^2 \equiv x^3 + a(t)x + b(t) \bmod p$. We let $a_t(p)$ be related to this count, and build a function $L(E_t,s) = \sum_n a_t(n) / n^s$ (we define how to get $a_t(n)$ from the prime values later); this is a generalization of the famous Riemann Zeta Function $\zeta(s) = \sum_n 1/n^s$. It turns out that the set of rational solutions to an elliptic curves forms a commutative group, and the Birch and Swinnerton-Dyer Conjecture states the order of vanishing of $L(E_t,s)$ at $s=1$ equals the geometric rank of the group.\footnote{We are using the normalization that the critical strip is $0 < {\rm Re}(s) < 2$, so the central point is $s=1/2$.} Further, the Generalized Riemann Hypothesis states that the non-trivial zeros of $L(E_t,s)$ have real part equal to 1.

It turns out that the moments of the Dirichlet coefficients $a_t(p)$ of the elliptic curve $L$-function are related to these problems. We define the $r$\textsuperscript{th} moment of the one-parameter family $E: y^2 = x^3 + a(T)x + b(T)$ to be \be A_{r,E}(p) \ := \ \frac1{p} \sum_{t \bmod p} a_t(p)^r \ee (we will see later that $a_t(p) = a_{t+\ell p}(p)$, and thus it suffices to sum over $t \bmod p$). By work of Rosen and Silverman \cite{RS} the first moment is related to the rank of the one-parameter family.\footnote{Assuming the Tate Conjecture \cite{Ta}, they prove a conjecture of Nagao \cite{Na} linking these two quantities.} Michel \cite{Mic} proved that if the elliptic curve doesn't have complex multiplication then the second moment is of size $p$, or better $pA_E(2) = p^2$ plus lower order terms of size $p^{3/2}, p, p^{1/2}$ and 1. In his thesis Miller \cite{Mi1, Mi2} discovered that in every family of elliptic curves where $pA_{2,E}(p)$ could be computed in closed form, the first lower order term that did not average to zero had a negative average. This has implications for the distribution of zeros of elliptic curves, connecting our work to the Generalized Riemann Hypothesis. This bias has been seen by many others, both in families of elliptic curves as well as other automorphic forms \cite{Mi1,Mi3}

We continue these investigations in two new directions. The first is prior studies concentrated on specific families where the computations could be done in closed form. Unfortunately, what this means is that perhaps the observed bias is merely a consequence of being such a special family that the computation can be done. While we do compute the bias in closed form for several new families, we also launch the first systematic investigation of general one-parameter families. In some cases the numerics do suggest closed form answers for the second moment's expansion, at least for primes in certain congruence classes. We also explore, for the first time, higher moments, to see if the biases persist.

For the rest of the introduction, after motivating our problem we quickly review some needed facts on elliptic curves; for more detail see \cite{Kn, ST, Ta}. We then describe in detail our results.

\subsection{Pythagorean Triples}

For thousands of years, there have been efforts to find integral and rational points on curves. We start with curves that are easy to analyze and also important: circles. It turns out that the set of rational points on a unit circle is closely related to primitive Pythagorean triples.

The Pythagorean theorem is one of the most important results in geometry, as it allows us to measure the distance between two points. Given any right triangle with sides $a, b$ and hypotenuse $c$, we have $a^2 + b^2 = c^2$; the triple is primitive if $\gcd(a,b,c) = 1$. It is easy to find some solutions, such as $(3,4,5)$ or $(5,12,13)$. Interestingly, one can generate all triples if we can find just one.

\begin{theorem}[Pythagorean triples] For all primitive Pythagorean triples $(a,b,c)$ there exist integers $p$ and $q$ with $p>q>0$ such that $a=2pq, b=p^2-q^2$ and $c=p^2+q^2$.
\end{theorem}

\begin{proof} Given $a^2 + b^2 = c^2$ we divide by $c^2$. Letting $x = a/c$ and $y=b/c$ we see our problem is equivalent to looking for rational solutions on the circle $x^2+y^2=1$. Note it is easy to find some points on this curve: $(\pm 1, 0)$ or $(0, \pm 1)$. We use the point $(0,-1)$ and draw a line through that point with slope $m$; see Figure \ref{fig:dimensionlinesquareA}.

\begin{figure}[t]
\begin{center}
\scalebox{1}{\includegraphics{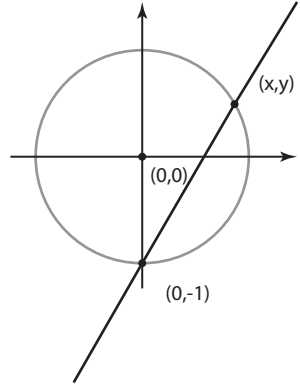}}
\caption{\label{fig:dimensionlinesquareA} A line through $(0,-1)$ with slope $m$, hitting the circle at  $(x,y)$.}
\end{center}
\end{figure}

The equation of our line is $y=mx-1$. It intersects the circle at $(0,-1)$ and one other point, say $(x,y)$. We find that by  simultaneously solving $x^2+y^2=1$ and $y=mx-1$. After some algebra we find \be x \ = \ \frac{2m}{m^2+1}, \ \ \  y \ = \ \frac{m^2-1}{m^2+1}. \ee

Note that $x$ and $y$ are rational numbers if and only if $m$ is a rational number; this is because we are solving $x^2 + (mx-1)^2 = 1$, which is a quadratic with rational coefficients and one root ($x = -1$) rational, forcing the other root to be rational. Thus the point on the unit circle is \be (x,y) \ = \ \left(\frac{2m}{m^2+1}, \frac{m^2-1}{m^2+1}\right).\ee As this is on the circle, its distance from the origin must be 1, and thus we see \be (2m)^2+(m^2-1)^2 \ = \ (m^2+1)^2. \ee

As our slope $m$ is a rational number, we may write $m = p/q$ with $\gcd(p,q) = 1$, and thus find \be (2pq)^2+(p^2-q^2)^2 \  =  \ (p^2+q^2)^2. \ee

We can also argue in the opposite direction; given a point $(x,y)$ on the circle with rational coordinates, the line from $(0, -1)$ to it has rational slope. Thus we see there is a 1-to-1 correspondence between points with rational coordinates and lines with rational slope $m$. \end{proof}

Building on our successful analysis of the circle, we could look at other quadratic forms (such as ellipses or hyperbolas). While there is a lot of interesting mathematics here, such as Pell's equation for certain hyperbola's, we instead move on to the next simplest polynomial to study: cubics.

\subsection{Elliptic Curves Preliminaries}

There are many degree three equations we can study. An elliptic curve is given by
\be y^2+a_1xy+a_3y \ = \ x^3+a_2x^2+a_4x+a_6, \ee with $a_1, a_2, a_3, a_4, a_6\in \mathbb{Z}$; we often impose some conditions so the curve is not degenerate. For example, if we had $y^2 = x^2(x-1)$ we could send $y \mapsto xy$ and ``reduce'' to studying $y^2 = x-1$, a parabola. Instead of having our coefficients in $\Z$ we could have them in a field; other common choices are the rationals $\Q$ as well as the integers modulo $p$. The integers and rationals are of characteristic zero, while the integers modulo $p$ are of characteristic $p$ ($p$ is always a prime unless stated otherwise).

When the elliptic curve is defined over field $K$ whose characteristic is neither 2 nor 3, we can convert the equation to a more useful form, called the Weierstrass form, which is of the form $y^2 = f(x)$ for some degree 3 polynomial $f$.

Consider the elliptic curve $E:y^2+a_1xy+a_3y=x^3+a_2x^2+a_4x+a_6$. Define
\bea b_2 & \ =  \ & a_1^2+4a_2\nonumber\\
b_4 &=& 2a_4+a_1a_3\nonumber\\
b_6 &=&a_3^2+4a_6\nonumber\\
b_8 &=& a_1^2a_6+4a_2a_6-a_1a_3a_4+a_2a_3^2-a_4^2\nonumber\\
c_4&=&b_2^2-24b_4\nonumber\\
c_6&=&-b_2^3+36b_2b_4-216b_6. \eea

The first simplification is that after some algebra we find our curve is equivalent to $y^2=4x^3+b_2x^2+2b_4x+b_6$ when ${\rm char}(K)\ne 2$ (we need to be able to invert 2 in the algebra). The second simplification gives us $y^2=x^3-27c_4x-54c_6$ when ${\rm char}(k)\ne 3$. Thus, frequently we study curves of the form $y^2 = x^3 + ax + b$.

The discriminant of the elliptic curve $y^2 = x^3 + ax + b$ is given by $\Delta=-16(4a^3+27b^2)$. The discriminant is nonzero when $x^3 + ax + b$ has three distinct roots. In order to avoid the cases when the curves degenerate, we impose the condition $4a^3+27b^2\ne 0$. In an elliptic function that has complex multiplication, the endomorphism ring is isomorphic to a subring in an imaginary quadratic field. The $j$-invariant of an elliptic curve is given by $j(E)=\frac{2^8 3^3 a^3}{4a^3+27b^2}$.

We let $E(\Q)$ denote the set of rational points on the elliptic curve $E$. Interestingly, this is not just a set, but also a group. We sketch the proof.

Let $P=(x_1,y_1)$ and $Q=(x_2,y_2)$ be two distinct points in $E(\Q)$ (if they are the same point we have to modify the construction slightly). Write the line connecting them by $y=\alpha x+\beta$; as our points have rational coordinates, $\alpha$ and $\beta$ are rational. The line intersects the elliptic curve at another point, unless the line is tangent to the curve; see Figure \ref{fig:dimensionlinesquareB}.

\begin{figure}[ht]
\begin{center}
\scalebox{1}{\includegraphics{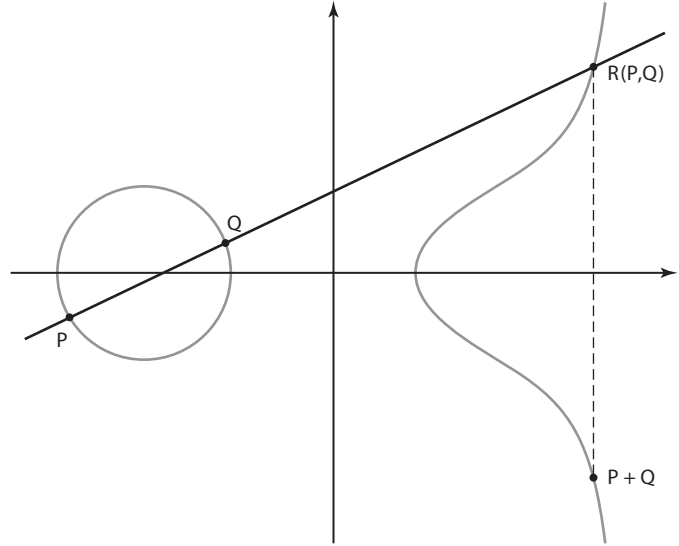}}
\caption{\label{fig:dimensionlinesquareB} The addition law of the elliptic curve.}
\end{center}
\end{figure}

To find the third point where the line hits the cubic, we substitute for $y$ with the equation of the line, and find
$x^3+ax+b=(\alpha x+\beta)^2$. As this cubic has rational coefficients and two rational roots ($x = x_1$ or $x_2$), the third root is also rational. Let $(x_3,y_3)$ be this point. We define $P+Q$ to be $(x_3,-y_3)$, which is the reflection of $(x_3,y_3)$ about the $x$-axis. With this definition, it turns out that the set of rational points forms a group. It is easy to see that it is a commutative group due to the geometric nature of the construction (the line through $P$ and $Q$ is the same as the line through $Q$ and $P$). Interestingly, what is hard is showing associativity: $(P+Q)+ R = P + (Q+R)$; this can be done through tedious algebra, or by appealing to high level results (the Rieman-Roch Theorem).

Now that we know $E(Q)$ is a group, it is natural to ask what is its structure. By the Mordell-Weil Theorem we can write it as $\Z^{r_g} \oplus \mathbb{T}_E$; here $r_g$ is the geometric rank of the group, and $\mathbb{T}_E$ is a finite torsion group. Mazur \cite{Ma} proved that there are only 15 possibilities for the torsion group: if $C_k$ is the cyclic group with $k$ elements, the possibilities are $C_n$ for $n \in \{1, \dots, 10, 12\}$ and the direct sum of $C_2$ with $C_2$, $C_4$, $C_6$ or $C_8$.

We end with a brief remark about why we care about points on rational curves. For years RSA was the gold standard for encryption, and was based on working with $(\Z/pq\Z)^\ast$ with $p, q$ distinct primes. The hope, which was realized through elliptic curves, was that using a more complicated group could lead to more secure encryption. It is thus an important question to know what types of groups can arise. To date the largest observed rank is 28, and some recent conjectures suggest (contrary to the initial thoughts) that the maximum rank might be bounded.

\subsection{Elliptic Curve $L$-function}

We now introduce the $L$-function of an elliptic curve, and discuss its connections with the geometric rank. The basic idea is that we use local information to build a global object, and then use the global object to get information about the local inputs. We illustrate this with two well-known example: Binet's formula for the Fibonacci numbers, and the Riemann zeta function.

\subsubsection{Generating Functions and Fibonacci Numbers}

For the Fibonacci numbers $F_{n+1} = F_n + F_{n-1}$ (with initial conditions $F_0 = 0, F_1 = 1$), we use the recurrence relation to build a generating function, and then use the closed form expression of the generating function to get Binet's formula for individual Fibonacci numbers.

\begin{theorem}[Binet's Formula] We have
\be F_n \ = \ \frac{1}{\sqrt{5}}\left[\left(\frac{1+\sqrt{5}}{2}\right)^n-\left(\frac{-1+\sqrt{5}}{2}\right)^n\right]. \ee
\end{theorem}

\begin{proof} We define the generating function by \be g(x)\ = \ \sum_{n>0} F_n x^n. \ee Using the recurrence relation and some algebra, we find
\bea \sum_{n\ge2} F_{n+1} x^{n+1} & \  = \ & \sum_{n\ge2} F_n x^{n+1}+\sum_{n\ge2} F_{n-1} x^{n+1} \nonumber\\
\sum_{n\ge3} F_{n} x^{n} &=&\sum_{n\ge2} F_n x^{n+1}+\sum_{n\ge1} F_{n} x^{n+2} \nonumber\\
\sum_{n\ge3} F_{n} x^{n} &=&x\cdot \sum_{n\ge2} F_n x^{n}+x^2\cdot \sum_{n\ge1} F_{n} x^{n}\nonumber\\
g(x)-F_1x-F_2x^2&=&x(g(x)-F_1x)+x^2g(x)\nonumber\\
g(x)&=&\frac{x}{1-x-x^2}. \eea

By partial fraction expansion,
\be g(x) \ = \ \frac{x}{1-x-x^2} \ = \ \frac{1}{\sqrt{5}}\left[\frac{(\frac{1+\sqrt{5}}{2})x}{1-(\frac{1+\sqrt{5}}{2})x}-\frac{(\frac{-1+\sqrt{5}}{2})x}{1-(\frac{-1+\sqrt{5}}{2})x}\right], \ee and then using the geometric series formula to expand each denominator yields the coefficient of $x^n$ is \be F_n\ = \ \frac{1}{\sqrt{5}}\left[\left(\frac{1+\sqrt{5}}{2}\right)^n-\left(\frac{-1+\sqrt{5}}{2}\right)^n\right]. \ee \end{proof}

The reason we are able to obtain a closed form expression for the individual Fibonacci numbers is that we have a closed form expression for the global object, the generating function. Before giving the elliptic curve $L$-function we first discuss the most basic example, the Riemann zeta function: \be \zeta(s) \ := \ \sum_{n=1}^\infty \frac1{n^s}, \ee which converges for ${\rm Re}(s) > 1$.

\subsubsection{Riemann Zeta Function}
\begin{theorem}[Euler product formula for the Riemann zeta function] For ${\rm Re}(s) > 1$ we have
\be \zeta(s) \ =\ \sum_{n=1}^{\infty} \frac{1}{n^s} \ =\ \prod_{p \ {\rm prime}} \left(1-\frac{1}{p^s}\right)^{-1}. \ee
\end{theorem}

\begin{proof} We sketch the argument, not worrying about convergence (which can be justified for such $s$); for details see for example \cite{MT-B}.
\bea \prod_{p\ {\rm prime}} \left(1-\frac{1}{p^s}\right)^{-1} & \ = & \left(1-\frac{1}{2^s}\right)^{-1} \left(1-\frac{1}{3^s}\right)^{-1} \left(1-\frac{1}{5^s}\right)^{-1}\cdots \nonumber\\
&=& \left(1+\frac{1}{2^s}+\frac{1}{4^s}\cdots\right) \left(1+\frac{1}{3^s}+\frac{1}{9^s}\cdots\right) \left(1+\frac{1}{5^s}+\frac{1}{25^s}\cdots\right)\cdots.\eea

By the Fundamental Theorem of Arithmetic, every positive integer can be written uniquely as a product of prime powers (with the primes in increasing order). Thus
\be \prod_{p\ {\rm prime}} \left(1-\frac{1}{p^s}\right)^{-1} \ = \ 1+\frac{1}{2^s}+\frac{1}{3^s}+\frac{1}{4^s}+\frac{1}{5^s}\cdots \ =\ \sum_{n=1}^{\infty} \frac{1}{n^s}. \ee
\end{proof}

The Riemann Zeta Function is useful as connects the integers and the primes. For example, as we take the limit as $s$ approaches 1 the sum-definition of the Riemann zeta function tends to the harmonic series, which diverges. Because this must equal the Euler product, we conclude that there are infinitely many primes (as otherwise the product would stay finite). With a little more work, we can also show that $\sum_{p} 1/p$ diverges and even obtain a growth rate on $\sum_{p \le x} 1/p$ from knowing $\sum_{n \le x} 1/n$ is $\log(x)$ plus lower order terms. Thus the primes are far more numerous than the perfect squares (whose reciprocal sum converges). This is just the start; by appealing to complex analysis (specifically contour integration) we obtain formulas connecting the number of primes up to $x$ with the distribution of zeros and poles of $\zeta(s)$ (though initially defined just for ${\rm Re}(s) > 1$, $\zeta(s)$ can be continued to be analytic everywhere save for a pole of residue 1 at $s=1$).

\subsubsection{Linear and Quadratic Legendre Sums}

Before we can define the elliptic curve $L$-function, we first need to introduce some terminology; we will also state some relations that play a central role in our studies.

\begin{definition}[Legendre symbol]
Let $p$ be an odd prime number, and let $a$ be an integer. We say $a$ is a quadratic residue modulo $p$ if it is congruent to a non-zero perfect square modulo $p$ and is a non-quadratic residue modulo $p$ otherwise. We define the Legendre symbol $\js{a}$ by \be \threecase{\js{a} \ =\ }{1}{if $a$ is a quadratic residue modulo $p$}{-1}{if $a$ is a non-quadratic residue modulo $p$}{0}{if $a \equiv 0$ modulo $p$.} \ee
\end{definition}

In the sums below, we write $\sum_{x (p)}$ to mean a sum over $x$ modulo $p$: thus it is $\sum_{x=0}^{p-1}$.

\begin{lemma}[Linear Legendre Sums] We have
\be \sum_{x(p)} \js{ax+b} \ =\ 0 \ee if $p \nmid a$; if $p|a$ the sum is just $p \js{b}$.
\end{lemma}

The proof is given in Appendix \ref{sec:linquadsums}.

\begin{lemma}[Quadratic Legendre Sums] Assume $a \not\equiv 0 \pmod p$. Then
\be \twocase{\sum_{x(p)} \js{ax^2+bx+c} \ = \ }{-\js{a}}{ if $p\nmid b^2-4ac$}{(p-1)\js{a}}{ if $p\mid b^2-4ac$.}\ee
\end{lemma}
The proof is given in Appendix \ref{sec:linquadsums}.
\\

Unfortunately, while we have great success with linear and quadratic legendre sums, there is no closed form solution for cubic and higher (for more on these and related sums, see \cite{BEW}). Our inability to handle in particular cubic sums makes the analysis of elliptic curves exceedingly difficult.

\subsubsection{Elliptic Curve $L$-functions}

Consider an elliptic curve $E: y^2 = x^3 + ax + b$. Let us look for pairs $(x,y)$ that solve this modulo $p$. If $x^3 + ax + b$ is a non-zero square modulo $p$ then there are two distinct $y$ that work for this $x$, if it is zero then only one $y$ works, while if it is not a square modulo $p$ no $y$ work. Thus the number of solutions is \be \sum_{x(p)} \left[\js{x^3+ax+b} + 1\right] \ = \ p + \sum_{x(p)} \js{x^3+ax+b}. \ee We define $a_E(p)$ to be $p$ minus the number of solutions mod p; Thus
\be a_E(p) \ := \ -\sum_{x(p)} \js{x^3+ax+b}.\ee Note that if we assume the $x^3+ax+b$ are as likely to be a quadratic residue as a non-residue, we would expect the number of solutions to be around $p$. Thus $a_E(p)$ measures the fluctuations in the number of solutions from the expected number. Hasse proved that $|a_E(p)| < 2 \sqrt{p}$.

\begin{remark}[Philosophy of Square-Root Cancellation]
The cancellation in Hasse's theorem is an example of the \emph{Philosophy of Square-Root Cancellation}: if you have a sum of $N$ signed terms all of order $B$, then we expect the size of the sum to be of order $B \sqrt{N}$ (up to sat a few powers of $\log N$). Perhaps the most famous example of this is the Central Limit Theorem: if we toss a fair coin $N$ times and record 1 for each head and -1 for each tail, we expect the sum to be zero with fluctuations of size $\sqrt{N}$. For elliptic curves, we have a sum of $p$ signed terms, each of size 1; the resulting sum is of size $\sqrt{p}$. \end{remark}

We now use this local data over the primes to build an elliptic curve $L$-function through an Euler product. We sketch the construction; in particular, we do not go into details on the conductor of the elliptic curve, $N_E$. For our purposes it is enough to know that it is an integer and $\chi_{N_E}(p) = 0$ if $p|N_E$ and 1 otherwise. The $L$-function of the elliptic curve $E$ is \be L(E,s) \ = \ \prod_p \left(1-\frac{a_E(p)}{p^{s}}+ \frac{\chi_{N_E}(p)}{p^{2s-1}}\right)^{-1}.\ee

The next conjecture, one of the seven Clay Millenial Problems,  gives one of the most important applications of the above.

\begin{conjecture}[Birch and Swinnerton-Dyer] The Taylor expansion of $L(E,s)$ at $s=1$ has the form $L(E,s)=c(s-1)^k + higher$ $order$ $terms$,
with $c\neq 0$ and $k$ the geometric rank of the group of rational solutions $E(\Q)$.
\end{conjecture}

In other words, the Birch and Swinnerton-Dyer Conjecture tells us that the order of vanishing of $L(E,s)$ at $s=1$ (the analytic rank) equals the geometric rank of the elliptic curve.

\subsection{Moments of the Dirichlet Coefficients of the Elliptic Curve $L$-functions}

We call the $a_E(p)$ the Dirichlet Coefficients of the elliptic curve $L$-function; if we have a one-parameter family we write $a_t(p)$ for $a_{E_t}(p)$. We study how these vary in a family.

Consider the one-parameter family of elliptic curves over $\Q(T)$ given by $E: y^2=x^3+a(T)x+b(T)$ where $a(T), b(T)$ are polynomials in $\mathbb{Z}[T]$; we often call $E$ an elliptic surface. The $r$-th moment of the Dirichlet coefficients is \be A_{r,E}(p) \ = \ \frac{1}{p}\sum_{t \bmod p} a_{t}(p)^r. \ee

\subsubsection{First Moment and Rank}
The first moment is \be A_{1,E}(p) \ = \ \frac{1}{p}\sum_{t \bmod p} a_{t}(p).\ee

\begin{theorem}[Rosen-Silverman] For an elliptic surface $E$ (a one-parameter family over $\Q(T)$),
\be
\lim_{x\to \infty}\frac{1}{x}\sum_{p\le x}\frac{A_{1,E}(p)\log p}{p} \ = \ {\rm rank}(E),
\ee where ${\rm rank}(E)$ is the rank of the group of rational points of $E$ over $\Q(T)$.
\end{theorem}

The Rosen-Silverman theorem tells us that a bias in the first moment is responsible for the geometric rank of elliptic curves. Given how important the rank of an elliptic curve is, it is thus natural to study higher moments to see if there are biases there as well, and if so what they might tell us about elliptic curves.

\subsubsection{Second Moment and the Bias Conjecture}
The second moment is \be A_{2,E}(p) \ = \ \frac{1}{p}\sum_{t \bmod p} a_{t}(p)^2.\ee Michel \cite{Mic} proved that if the elliptic curve does not have $j(T)$ constant then \be A_{2,E}(p) \ = \ p^2+ O(p^{3/2}), \ee where the lower order terms have size $p^{3/2}, p, p^{1/2}$ and $1$. We are using big-Oh notation: $f(x) = O(g(x))$ if for all $x$ sufficiently large there is some $C$ such that $|f(x)| \le C |g(x)|$.

In his thesis Miller \cite{Mi1, Mi2} studied several families where $A_{2,E}(p)$ could be computed in closed form. In all of these, the first lower term that did not average to zero had a negative average. He further showed that such a negative bias has implications in the distribution of zeros of the elliptic curve $L$-function near the central point. In particular, it helped explain a small amount of the observed excess rank, thus highlighting how important these terms are. He conjectured that this bias exists in every family, and so far this has been observed in every family studied to date \cite{HKLM, MMRW, Mi4}.

\begin{conjecture}[Negative Bias Conjecture (Miller)]
The largest lower order term in the second moment expansion of a one-parameter family that does not average to 0 has a negative average.
\end{conjecture}

The main idea in previous work has been to look at the triple sum \be \frac1{p} \sum_{t(p)} \sum_{x(p)} \sum_{y(p)} \js{x^3 + a(t)x + b}\js{y^3+a(t)y+b(t)}. \ee There is no hope in executing the $x$ or $y$ sum first; the idea is to switch orders and sum over $t$. By carefully choosing $a(t)$ and $b(t)$ one is able to get a closed form solution for this $t$-sum, and if one is fortunate the resulting $x$ and $y$ sums can also be done in closed form. This is typically the case if, after changing variables in the sums, we have a quadratic in $t$ sum. Then by our results on quadratic Legendre sums, its value is a function of the discriminant; here the discriminant will be a function of $x$ and $y$. In many cases we can carefully choose the polynomials so we can easily determine when the discriminant is zero modulo $p$, and thus we \emph{can} evaluate the resulting $x$ and $y$ sums, and obtain closed form expressions for the second moment.

Of course, it is possible that the bias is present \emph{because} the families investigated using the above techniques are clearly not generic. It is thus important to study all families of elliptic curves, and not just special families where the sums can be executed. All previous work has been focused on trying to find as general as possible families where we can execute these sums in closed form. This has been done successfully for many one and two-parameter families, see for example \cite{ACFKKMMWWYY, HKLM, MMRW, Mi3, Mi4, Wu}; note biases in two-parameter families were first studied by Wu \cite{Wu} last year in her S.-T. Yau High School Science project. This work continues that program by moving in \emph{two} new directions.

First, while we do look at some additional families, we give the first systematic investigation of one-parameter families of given rank. Thus most of the time we are unable to obtain closed form expressions, and instead we have to try to determine if a bias exists or not. Second, we look at whether or not the biases persist to higher moments. This introduces a tremendous computational challenge, as the bias conjecture states that the first lower order term that does not average to zero has a negative average; if there is a larger lower order term that does average to zero, it will drown out this bias and make it impossible to see numerically.

For example, we will show later the fourth moment is of size $2p^3$ with lower order terms of size $p^{5/2}, p^2$ and so on. Assume we have a family where there is a term of size $p^{5/2}$ but it averages to zero, and there is a term of size $p^2$ that has a negative average. If we look at the observed fourth moment minus the main term ($2p^2$) and divide by $p^{5/2}$, the contribution from the $p^{5/2}$ term will average to zero. Similar to our coin tossing example from before and the Central Limit Theorem, it will oscillate around zero and if we sum $P$ such primes, the average will be of size $\sqrt{P}/P = 1/\sqrt{P}$; the $\sqrt{P}$ comes from summing $P$ signed terms and assuming square-root cancellation, while the $P$ in the denominator comes from taking an average. Meanwhile, the contribution from the lower order term of size $p^2$ with a negative bias is hidden; since we are dividing by $p^{5/2}$ each term here will be of size $1/\sqrt{p}$, and thus its average will be of order $\sqrt{P}/P$ or $1/\sqrt{P}$.

This creates tremendous challenges in our studies. We will investigate higher moments, and sometimes we will still be able to see the bias. This can happen for example if there is no $p^{5/2}$ term, and we can thus divide by $p^2$. While these are not proofs, these results indicate interesting phenomena that are worth further study. In other families we can at least try to show that the numerics \emph{suggest} that the first lower order term averages to zero. We can do this by looking at many averages over consecutive blocks. For example, it takes a few days to compute the results for the first 1000 primes in a family. We can break into 20 blocks of 50 consecutive primes, and by the Central Limit Theorem each one of these blocks should be of size 0 with fluctuations of size $\sqrt{50}$ if it truly averages to zero. We can investigate if this happens, as well as look and see if roughly half the blocks have a small positive and half a small negative average. Additionally, we can look at the grand average of the first 1000 primes; if the first lower order term does average to zero we expect to see something on the order of $1/\sqrt{1000} \approx .032$ (thus anything two or three times this, positive or negative, is suggestive but not a proof of averaging to zero).

\subsection{Summary of Results}

We have systematically investigated general one-parameter families. For some families, we are able to find the second moments expansions by separating primes into different congruence classes, which suggests that there is often a closed-form polynomial expression. We are then able to prove the results mathematically. The following table summarizes the numerical results for the second moments' expansions of the families we studied; see Figure \ref{fig:dimensionlinesquareD}.
\begin{figure}[ht]
\begin{center}
\scalebox{1.5}{\includegraphics{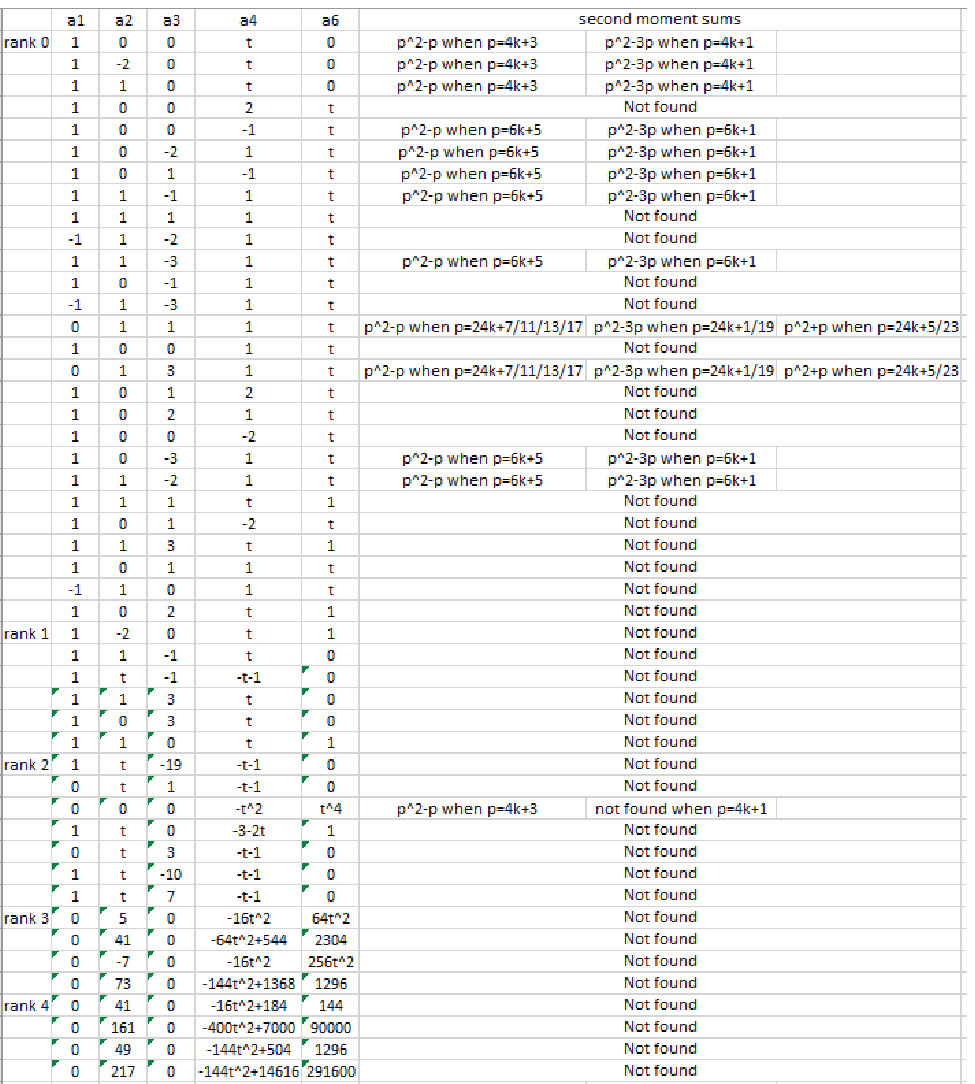}}
\caption{\label{fig:dimensionlinesquareD} Systematic investigation for second moments sums.}
\end{center}
\end{figure}
\ \\

Below is a summary of the first and second moments for some families where we are able to prove closed form expressions for the first two moments. We give the proofs in the next section; the arguments there are representative of the ones needed for all the families.

\ \\
\noindent \textbf{Family: $y^2=4x^3+ax^2+bx+c+dt$:}
\begin{itemize}
\item First moment: $A_{1,\varepsilon(p)}=0$.
\item Second moment:
\begin{equation}
A_{2,\varepsilon(p)} \ =\
\Bigg\{ {\  p^2-p-p\cdot \js{-48}-p\cdot \js{a^2-12b} \ \
\mbox{if} \  a^2-12b\ne 0 \atop p^2-p+p(p-1)\js{-48} \ \mbox{otherwise.}}
\end{equation}
\end{itemize}

\ \\
\noindent \textbf{Family: $y^2=4x^3+(4m+1)x^2+n\cdot tx$:}
\begin{itemize}
\item First moment: $A_{1,\varepsilon(p)}=0$.
\item Second moment: \begin{equation} A_{2,\varepsilon(p)}\ =\
\Bigg\{ {\ p^2-3p \ \
\mbox{if} \  p=4k+1 \atop p^2-p \ \mbox{if} \ p=4k+3.}
\end{equation}
\end{itemize}

\ \\
\noindent \textbf{Family: $y^2=x^3-t^2x+t^4$:}
\begin{itemize}
\item First moment: $A_{1,\varepsilon(p)} \ = \ -2p$.
\item Second moment:
\begin{equation}
A_{2,\varepsilon(p)}=p^2-p-p\cdot \js{-3}-p\cdot \js{12}-\sum_{x(p)} \sum_{y(p)} \js{x^3-x} \js{y^3-y}.
\end{equation}
\end{itemize}

\ \\

For families that we are not able to find closed-form expressions, we calculated the average bias of the second moment sums for the first $1000$ primes; see Figure \ref{4_6moments_1}.

In addition, we studied biases in the fourth and sixth moments, the first time such investigations have been done. We calculated the average bias of the fourth and sixth moment sums for the first $1000$ primes; see Figure \ref{4_6moments_1}. For some families, fluctuations of larger lower order terms that (we believe) average to zero drown out biases in even smaller lower order terms, and unfortunately make it impossible to see the biases numerically. We can however gather data supporting that the first lower order term averages to zero for families with lower ranks, and we will discuss the statistical tests needed to glean that from our observations.

\begin{figure}[ht]
\begin{center}
\scalebox{1}{\includegraphics{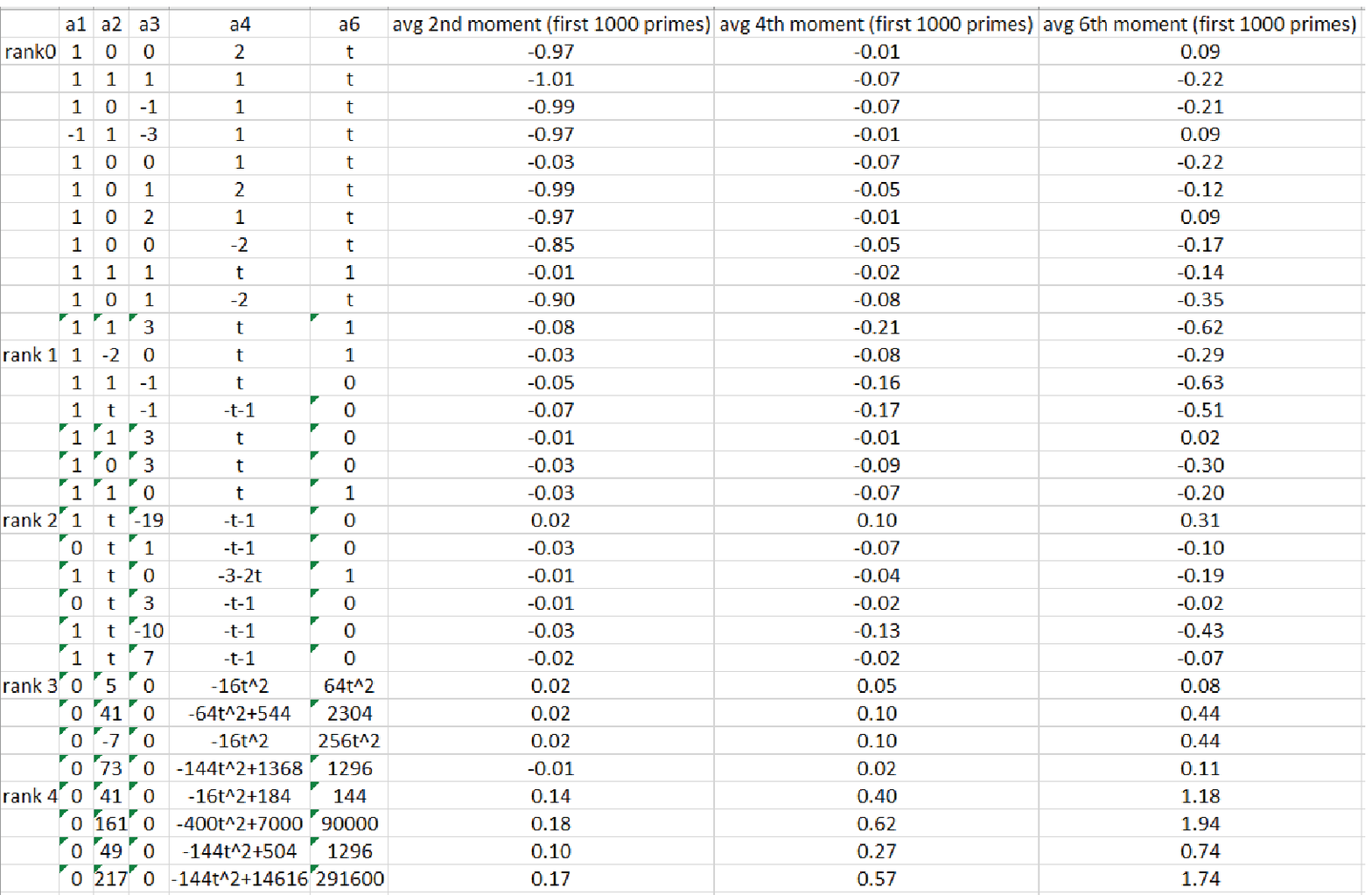}}
\caption{\label{4_6moments_1}Numerical data for the average biases of 2nd, 4th and 6th moments sums.}
\end{center}
\end{figure}

Through numeric computations for these families, we believe that the rank of the elliptic curves might play a role in determining the bias. Our data suggests the following (for details see Sections \ref{sec:two} and \ref{sec:three}).

\begin{itemize}

\item The rank $0$ and rank $1$ families have negative biases in their second moment sums. However, it is likely that higher rank families (${\rm rank}(E(\Q))\ge4$) have positive biases.

\item For the fourth moment, we also believe that the rank $0$ and rank $1$ families have negative biases. We see this in some families; in others we see the presence of terms of size $p^{5/2}$ whose behavior is consistent with their averaging to zero. Thus the data at least suggests a weaker version of the bias conjecture: the first lower order term does not have a positive bias for lower rank families. On the other hand, higher rank families (${\rm rank}(E(\Q))\ge4$) appear to have positive biases. 

\item The sixth moment results are similar to the fourth moment. The results are consistent with either a negative bias, or a leading term (now of size $p^{7/2}$) averaging to zero for lower rank families. Our data also suggests that higher rank families have positive biases. 

\item For the odd moments, the coefficients of the leading term vary with the primes. Our data suggests that the average value of the main term for the $2k+1$-th moment is $-C_{k+1} {\rm rank}(E(\Q)) p^{k+1}$, where $C_n$ is the n-th term of the Catalan numbers
\end{itemize}

\section{Biases in second moments of elliptic curve families}\label{sec:two}

In this section, we explore the existence of negative bias in second moments sums through numerical evidence for general elliptic curve families and computation of the bias in closed form for several new families. All previous research on the negative bias conjecture of second moment sums only studied specific, special families where the resulting sums could be done in closed form. Therefore, it is possible that the negative bias conjecture does not apply to all elliptic curves, and is an artifact of looking at carefully chosen, non-generic families.

\subsection{Systematic investigation for the second moments sums}
We have systematically investigated general one-parameter families. For some families we are able to find polynomial expressions for the second moment expansions by separating the primes into different congruence classes, which suggests that there is a closed-form expression for some families (from previous work of Miller \cite{Mi3} we know however that there are families where the $p^{3/2}$ term arises, and these are not polynomials). The following table summarizes the numerical results for the second moment expansions of the families we studied; see Figure \ref{fig:dimensionlinesquareC}.

\begin{figure}[ht]
\begin{center}
\scalebox{1.5}{\includegraphics{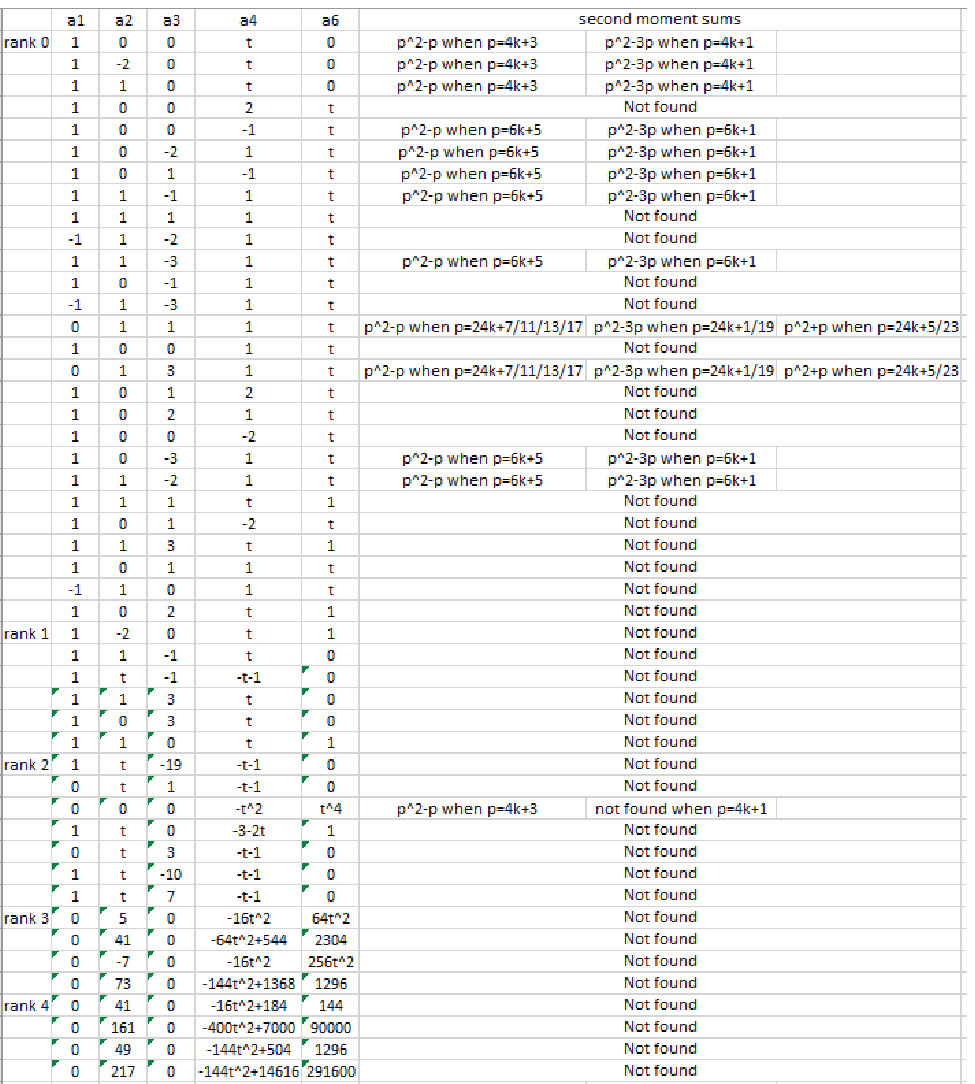}}
\caption{\label{fig:dimensionlinesquareC} Systematic investigation for second moments sums.}
\end{center}
\end{figure}

For all the families that we are not able to find the second moments expansions, we cannot find formulas for primes in congruence classes modulo $2^4\cdot 3^3\cdot 5$, which strongly suggests that there isn't a closed-form polynomial for their second moment sums. We chose to study this modulus as many properties of elliptic curves depend on powers of 2 and 3. Further, in the analysis above the largest modulus needed was $24 = 2^3 \cdot 3$, suggesting that $2^4\cdot 3^3\cdot 5$ is a reasonably safe choice.

In all the cases above where the numerics suggest closed-forms answers, we are able to prove it mathematically.

We now turn from numerical experimentation to theoretical calculation, and discuss some representative calculations of ours.

\subsection{First and second moments of the family $y^2=4x^3+ax^2+bx+c+dt$}

\begin{lemma}The first moment of the family $y^2=4x^3+ax^2+bx+c+dt$ is 0.
\end{lemma}

\begin{proof}
For all $p>4d$, send $t$ to $4d^{-1}t$: Thus
\be \sum_{t(p)} \js{dt}=\sum_{t(p)} \js{4t} .\ee
Therefore,
\bea A_{1,\varepsilon(p)}&\ = \ & -\sum_{t(p)} \sum_{x(p)}  \js{4x^3+ax^2+bx+4t+c}\nonumber\\
&\ = \ & -\sum_{x(p)} \sum_{t(p)} \js{4t+4x^3+ax^2+bx+c}.\eea
As $p\not|4$ when $p\ne 2$, the t-sum vanishes by linear sum theorem, and
\be A_{1,\varepsilon(p)} \ = \ 0.  \ee
By the Rosen-Silverman Theorem, this is a rank 0 family. \end{proof}

\begin{lemma} The second moment of the family $y^2=4x^3+ax^2+bx+c+dt$ is
\begin{equation}
\twocase{A_{2,E}(p)\ =\ }{p^2-p-p\cdot \js{-48}-p\cdot \js{a^2-12b}}{if $a^2-12b\ne 0$}{p^2-p+p(p-1)\js{-48}}{otherwise.}
\end{equation}
\end{lemma}

\begin{proof} We have
\bea A_{2,E}(p)&\ = \ & \sum_{t(p)} \sum_{x(p)} \sum_{y(p)} \js{4x^3+ax^2+bx+4t+c} \js{4y^3+ay^2+by+4t+c} \nonumber\\ m(x)&\ = \ &4x^3+ax^2+bx+c \nonumber\\
n(y)&\ = \ &4y^3+ay^2+by+c \nonumber\\
A_{2,E}(p)&\ = \ & \sum_{t(p)} \sum_{x(p)} \sum_{y(p)}  \js{16t^2+4(m+n)t+mn}.
\eea
The discriminant of $16t^2+4(m+n)t+mn$ is
\bea
\Delta_t(x,y)&\ = \ &16(m+n)^2-64mn \nonumber\\
&\ = \ &16(m-n)^2 \nonumber\\
\delta^2&\ = \ &\Delta_t(x,y) \nonumber\\
\delta&\ = \ &4(m-n)\nonumber\\
&\ = \ &4(4x^3+ax^2+bx+c-4y^3-ay^2-by-c)\nonumber\\
&\ = \ &4(x-y)(4x^2+4xy+4y^2+ax+ay+b).
\eea

If $p|\delta$, then $p|x-y$ or $p|4x^2+4xy+4y^2+ax+ay+b$. $x, y$ range from $0$ to $p-1$, so $p|x-y$ exactly $p$ times. By the quadratic formula mod $p$, $4x^2+4xy+4y^2+ax+ay+b\equiv 4y^2+(4x+a)y+4x^2+ax+b\equiv 0 \pmod p$ when $y=\frac{-4x-a\pm \sqrt{\Delta_y}}{8}$  ($\Delta_y$ is the discriminant of the polynomial  $4y^2+(4x+a)y+4x^2+ax+b$ in terms of y):
\bea
\Delta_y&\ =\ &(4x+a)^2-4\cdot 4(4x^2+ax+b)\nonumber\\
&\ = \ &-48x^2-8ax+a^2-16b.
\eea
If $\Delta_y$ is a non-zero square mod $p$, there are two solutions. If $\Delta_y$ is 0 mod $p$, there is one solution. If $\Delta_y$ is not a square mod $p$, there is no solution.
\\
The number of pairs of $x,y$ such that $p|4x^2+4xy+4y^2+ax+ay+b$ is
\be
\sum_{x(p)} 1+\js{-48x^2-8ax+a^2-16b}=p+\sum_{x(p)} \js{-48x^2-8ax+a^2-16b}.
\ee
The discriminant of $-48x^2-8ax+a^2-16b$ is
\bea
\Delta_x&\ = \ &(8a)^2-4\cdot (-48)(a^2-16b)\nonumber\\
&\ = \ &256a^2-3072b\nonumber\\
&\ = \ &256(a^2-12b).
\eea
\ \\

We break into cases, depending on the value of the discriminant.

\noindent \emph{Case 1: $a^2-12b\ne 0$}:
By the Quadratic Legendre Sum Theorem, if $p\not |256(a^2-12b)$,
\be \sum_{x(p)} \js{-48x^2-8ax+a^2-16b}=-\js{-48}. \ee
The number of pairs of $x,y$ such that $p|4x^2+4xy+4y^2+ax+ay+b$ is
\be
p+\sum_{x(p)} \js{-48x^2-8ax+a^2-16b}=p-\js{-48}.
\ee

The cases that we double count $x=y$ and $p|4x^2+4xy+4y^2+ax+ay+b$ is
\be
4x^2+4xy+4y^2+ax+ay+b\equiv 12y^2+2ay+b\equiv0 \pmod p. \label{acdde1}
\ee
The discriminant of $12y^2+2ay+b$ is
\be
\Delta_y=(2a)^2-4\cdot 12b=2^2\cdot (a^2-12b).
\ee

By the quadratic formula mod $p$, the number of solutions is computable, depending on $a^2-12b$.

\ \\
The number of solutions to (\ref{acdde1}) is
\be
1+\js{2^2\cdot (a^2-12b)}=1+\js{a^2-12b}.
\ee
Therefore, the total number of times that $p|(x-y)(4x^2+4xy+4y^2+ax+ay+b)$ is the number of times $p|x-y$ plus the number of times $p|4x^2+4xy+4y^2+ax+ay+b$ minus the cases that we double count.
\ \\
The number of times that $p|(x-y)(4x^2+4xy+4y^2+ax+ay+b)$ is
\be
p+p-\js{-48}-1-\js{a^2-12b}=2p-1-\js{-48}-\js{a^2-12b}.
\ee
The number of times that $p\not |(x-y)(4x^2+4xy+4y^2+x+y-4)$ is
\be
p^2-\left(2p-1-\js{-48}-\js{a^2-12b}\right)=p^2-2p+1+\js{-48}+\js{a^2-12b}.
\ee
By Quadratic Legendre Sum Theorem,
\bea
A_{2,E}(p)&\ = \ &(p-1)\left[2p-1-\js{-48}-\js{a^2-12b}\right]-\left[p^2-2p+1+\js{-48}+\js{a^2-12b}\right]\nonumber\\
&\ = \ &p^2-p-p\cdot \js{-48}-p\cdot \js{a^2-12b}.
\eea

\ \\

\noindent \emph{Case 2: $a^2-12b=0$}: By the Quadratic Legendre Sum Theorem, since $p|0$, we have
\be
\sum_{x(p)} \js{-48x^2-8ax+a^2-16b}=(p-1)\js{-48}.
\ee
The number of pairs of $x,y$ such that $p|4x^2+4xy+4y^2+ax+ay+b$ is
\be
p+\sum_{x(p)} \js{-48x^2-8ax+a^2-16b}=p+(p-1)\js{-48}.
\ee

The cases that we double count $x=y$ and $p|4x^2+4xy+4y^2+ax+ay+b$ is
\be
4x^2+4xy+4y^2+ax+ay+b\equiv 12y^2+2ay+b\equiv0 \pmod p. \label{acdde}
\ee
The discriminant of $12y^2+2ay+b$ is
\be
\Delta_y=(2a)^2-4\cdot 12b=2^2\cdot (a^2-12b).
\ee

By the quadratic formula mod $p$, the number of solutions is computable, depending on $a^2-12b$.

The number of solutions to (\ref{acdde}) is
\be
1+\js{2^2\cdot (a^2-12b)}=1.
\ee
Therefore, the total number of times that $p|(x-y)(4x^2+4xy+4y^2+ax+ay+b)$ is the number of times $p|x-y$ plus the number of times $p|4x^2+4xy+4y^2+ax+ay+b$ minus the cases that we double count.
\ \\
The number of times that $p|(x-y)(4x^2+4xy+4y^2+ax+ay+b)$ is
\be
p+p+(p-1)\js{-48}-1=2p-1+(p-1)\js{-48}.
\ee
The number of times that $p\not |(x-y)(4x^2+4xy+4y^2+x+y-4)$ is
\be
p^2-\left(2p-1+(p-1)\js{-48}\right)=p^2-2p+1-(p-1)\js{-48}.
\ee
    
\bea
A_{2,E}(p)&\ = \ &(p-1)\left[2p-1+(p-1)\js{-48}\right]-\left[p^2-2p+1-(p-1)\js{-48}\right]\nonumber\\
&\ = \ &p^2-p+p(p-1)\js{-48}.
\eea

Thus we have shown \begin{equation}
A_{2,E}(p) =
\Bigg\{ {\  p^2-p-p\cdot \js{-48}-p\cdot \js{a^2-12b} \ \
\mbox{if} \  a^2-12b\ne 0 \atop p^2-p+p(p-1)\js{-48} \ \mbox{otherwise.}}
\end{equation}
\end{proof}

\begin{lemma}When $a^2-12b$ is a non-zero square, the second moment of the family $y^2=4x^3+ax^2+bx+c+dt$ is
\begin{equation}
A_{2,E}(p)=p^2-2p-p\cdot \js{-48}
\end{equation}
\end{lemma}

\ \ \\ \ \\ \ \\
Nine of the families that we investigated through numerical computations are special cases of this elliptic curve family.
\subsubsection{$y^2+xy=x^3-x+t$}
($a_1=1$ $a_2=0$ $a_3=0$ $a_4=-1$ $a_6=t$).
\bea
y^2&\ = \ &4x^3+x^2-4x+4t. \nonumber\\
a^2-12b&\ = \ &1^2-12\cdot(-4)=7^2.
\nonumber\\
A_{1,\varepsilon(p)}&\ = \ &0.
\nonumber\\
A_{2,E}(p)&\ = \ &p^2-2p-p\cdot \js{-48}.
\eea

\subsubsection{$y^2+xy-2y=x^3+x+t$}
($a_1=1 a_2=0 a_3=-2 a_4=1 a_6=t$).
\bea
y^2&\ = \ &4x^3+x^2+4t+4.
\nonumber\\
a^2-12b&\ = \ &1^2-12\cdot(0)=1^2.
\nonumber\\
A_{1,\varepsilon(p)}&\ = \ &0.
\nonumber\\
A_{2,E}(p)&\ = \ &p^2-2p-p\cdot \js{-48}.
\eea

\subsubsection{$y^2+xy+y=x^3-x+t$}
($a_1=1$ $a_2=0$ $a_3=1$ $a_4=-1$ $a_6=t$).
\bea
y^2&\ = \ &4x^3+x^2-2x+4t+1.
\nonumber\\
a^2-12b&\ = \ &1^2-12\cdot(-2)=5^2.
\nonumber\\
A_{1,\varepsilon(p)}&\ = \ &0.
\nonumber\\
A_{2,E}(p)&\ = \ &p^2-2p-p\cdot \js{-48}.
\eea

\subsubsection{$y^2+xy-y=x^3+x^2+x+t$}
($a_1=1$ $a_2=1$ $a_3=-1$ $a_4=1$ $a_6=t$).
\bea
y^2&\ = \ &4x^3+5x^2+2x+4t+1.
\nonumber\\
a^2-12b&\ = \ &5^2-12\cdot(2)=1^2.
\nonumber\\
A_{1,\varepsilon(p)}&\ = \ &0.
\nonumber\\
A_{2,E}(p)&\ = \ &p^2-2p-p\cdot \js{-48}.
\eea

\subsubsection{$y^2+xy-3y=x^3+x^2+x+t$}
($a_1=1$ $a_2=1$ $a_3=-3$ $a_4=1$ $a_6=t$).
\bea
y^2&\ = \ &4x^3+5x^2-2x+4t+9.
\nonumber\\
a^2-12b&\ = \ &5^2-12\cdot(-2)=7^2.
\nonumber\\
A_{1,\varepsilon(p)}&\ = \ &0.
\nonumber\\
A_{2,E}(p)&\ = \ &p^2-2p-p\cdot \js{-48}.
\eea

\subsubsection{$y^2+xy-3y=x^3+x+t$}
($a_1=1$ $a_2=0$ $a_3=-3$ $a_4=1$ $a_6=t$).
\bea
y^2&\ = \ &4x^3+x^2-2x+4t+9.
\nonumber\\
a^2-12b&\ = \ &1^2-12\cdot(-2)=5^2.
\nonumber\\
A_{1,\varepsilon(p)}&\ = \ &0.
\nonumber\\
A_{2,E}(p)&\ = \ &p^2-2p-p\cdot \js{-48}.
\eea

\subsubsection{$y^2+xy-2y=x^3+x^2+x+t$}
($a_1=1$ $a_2=1$ $a_3=-2$ $a_4=1$ $a_6=t$).
\bea
y^2&\ = \ &4x^3+5x^2+4t+4.
\nonumber\\
a^2-12b&\ = \ &5^2-12\cdot(0)=5^2.
\nonumber\\
A_{1,\varepsilon(p)}&\ = \ &0.
\nonumber\\
A_{2,E}(p)&\ = \ &p^2-2p-p\cdot \js{-48}.
\eea

\subsubsection{$y^2+y=x^3+x^2+x+t$}
($a_1=0$ $a_2=1$ $a_3=1$ $a_4=1$ $a_6=t$).
\bea
y^2&\ = \ &4x^3+4x^2+4x+4t+1.
\nonumber\\
a^2-12b&\ = \ &4^2-12\cdot(4)=-32.
\nonumber\\
A_{1,\varepsilon(p)}&\ = \ &
0.
\nonumber\\
A_{2,E}(p)&\ = \ &p^2-p-p\cdot \js{-48}-p\cdot \js{-32}.
\eea

\subsubsection{$y^2+3y=x^3+x^2+x+t$}
($a_1=0$ $a_2=1$ $a_3=3$ $a_4=1$ $a_6=t$).
\bea
y^2&\ = \ &4x^3+4x^2+4x+4t+9.
\nonumber\\
a^2-12b&\ = \ &4^2-12\cdot(4)=-32.
\nonumber\\
A_{1,\varepsilon(p)}&\ = \ &0.
\nonumber\\
A_{2,E}(p)&\ = \ &p^2-p-p\cdot \js{-48}-p\cdot \js{-32}.
\eea

\subsection{First and second moments of the family $y^2=4x^3+(4m+1)x^2+n\cdot tx$}
\begin{lemma}The first moment of the family $y^2=4x^3+(4m+1)x^2+n\cdot tx$ is 0.
\end{lemma}

\begin{proof} For all $p>4n$, send $t$ to $4n^{-1}t$. We have
\bea
\sum_{t(p)} \js{nt}&\ = \ &\sum_{t(p)} \js{4t}.\nonumber\\
A_{1,\varepsilon(p)}&\ = \ &
-\sum_{t(p)} \sum_{x(p)} \js{4x^3 + (4m+1)x^2 + 4tx}\nonumber\\
&\ = \ &-\sum_{x=1}^{p-1} \sum_{t(p)} \js{4xt + 4x^3 + (4m+1)x^2}.
\eea
As $p\not|4x$ when $p\ne 2$,  the $t$-sum vanishes by the linear Legendre sum theorem, and hence
\be
A_{1,\varepsilon(p)} \ = \ 0
\ee as claimed. \end{proof}

\begin{lemma} The second moment of the family $y^2=4x^3+(4m+1)x^2+n\cdot tx$ is
\begin{equation}
A_{2,E}(p)\ =\
\Bigg\{ {\ p^2-3p \ \
\mbox{{\rm if}} \  p=4k+1 \atop p^2-p \ \mbox{{\rm if}} \ p=4k+3.}
\end{equation}
\end{lemma}

\begin{proof}
\bea
A_{2,E}(p)&\ =\ & \sum_{t(p)} \sum_{x(p)} \sum_{y(p)} \js{4x^3 + (4m+1)x^2 + 4tx} \js{4y^3 + (4m+1)y^2 + 4ty}\nonumber\\
& = & \sum_{t(p)} \sum_{x(p)} \sum_{y(p)} \js{xy} \nonumber\\ & & \cdot \js{16t^2+4[4x^2+(4m+1)x+4y^2+(4m+1)y]t+[4x^2+(4m+1)x][4y^2+(4m+1)y]}.\nonumber\\
a(x)&\ = \ &4x^2+(4m+1)x. \nonumber\\
b(y)&\ = \ &4y^2+(4m+1)y.\nonumber\\
A_{2,E}(p)&\ = \ & \sum_{t(p)} \sum_{x(p)} \sum_{y(p)}  \js{16xy\cdot t^2+4xy(a+b)\cdot t+xy\cdot ab}.
\eea

We look at the discriminant of our polynomial in $t$:
\bea
\Delta&\ = \ &16x^2y^2[(a+b)^2-4ab]\nonumber\\
&\ = \ &16x^2y^2(a-b)^2.\\
\delta^2&\ = \ &\Delta \nonumber\\
\delta&\ = \ &4xy(a-b) \nonumber\\
&\ = \ &4xy[4x^2+(4m+1)x-4y^2-(4m+1)y] \nonumber\\
&\ = \ &4xy(x-y)(4x+4y+4m+1).
\eea
If $p|\delta$, then $xy=0$ or $x=y$ or $p\mid 4x+4y+4m+1$. We can ignore the cases when $xy=0$ because $\js{16xy}=0$, so it does not contribute to the second moment sum.
\ \ \\ \ \\ \ \\

We need to count the number of pairs of $x,y$ ranging from 0 to $p-1$ such that $p|4x+4y+4m+1$.
Thus we must consider choices where $0\le 4x+4y+1\le 8p-7$.

For $p>4m$, $0\le 4x+4y+4m+1<9p$.
\\
If $p|4x+4y+4m+1$, then $4x+4y+4m+1=0,p,2p,...,7p,8p$.
\\
There is no solution for $4(x+y+m)+1=0,2p,4p,6p,8p$.
\ \\ \ \\
\noindent \emph{If $4(x+y+m)+1=p$,}
\\
\noindent Case 1: $p=4k+1$: We must solve
\be
x+y+m \  = \ k\nonumber\\
\ee
\be
\systeme*{x=0,y=k-m}, . . . , \systeme*{x=0,y=k-m.}
\ee
There are $k-m+1$ solutions.
\\
If $k$ and $m$ are both even or both odd, there is one solution such that $x=y$.
\noindent Case 2: $p=4k+3$: No solution.
\ \ \\ \ \\ \ \\

\noindent \emph{If $4(x+y+m)+1=3p$,}
\\
Case 1: $p=4k+1$: We have $4(x+y+m)=12k+2$: no solution.
\ \\ 
Case 2: $p=4k+3$: We have to solve
\be
x+y+m \ =\ 3k+2\nonumber\\
\ee
\be
\systeme*{x=0,y=3k+2-m} , \ \ . . . \ \ , \systeme*{x=3k+2-m,y=0.}
\ee
There are $3k+3-m$ solutions.
\\
If $k$ and $m$ are both even or both odd, there is one solution such that $x=y$.
\ \ \\ \ \\ \ \ \\ \ \\

\noindent \emph{If $4(x+y+m)+1=5p$,}
\\
Case 1: $p=4k+1$: We have to solve
\be
x+y+m \ = \ 5k+1(x,y<4k+1)
\ee
\be
\systeme*{x=k+1-m,y=4k} , . . . , \systeme*{x=4k,y=k+1-m.}
\ee
There are $3k+m$ solutions.
\\
If $k$ and $m$ have opposite parity, there is one solution such that $x=y$.
\ \\ \ \\
Case 2: $p=4k+3$: $4(x+y+m)=20k+14$: no solution
\ \ \\ \ \\ \ \\

\noindent \emph{If $4(x+y+m)+1=7p$,}
\\
Case 1: $p=4k+1$: $4(x+y+m)=28k+6$: no solution
\ \\ 
Case 2: $p=4k+3$: we have to solve
\be
x+y+m \ =\ 7k+5 (x,y < 4k+3)
\ee
\be
\systeme*{x=3k+3-m,y=4k+2} ,\ \ . . . \ \ , \systeme*{x=4k+2,y=3k+3-m.}
\ee
There are $k+m$ solutions.
\\
If $k$ and $m$ have opposite parity, there is one solution such that $x=y$.
\ \\ \ \\

Combining the above, the total number of solutions is
\begin{itemize}
\item if $p=4k+1$: $(k-m+1)+3k+m=4k+1=p$,

\item if $p=4k+3$: $(3k+3-m)+k+m=4k+3=p$.
\end{itemize}

There are always $p$ times that $p|4x+4y+4m+1$ when $0\le x,y \le p-1$.
\\
There is always exactly one time such that $p\mid 4x+4y+4m+1$ and $x=y$.
\ \ \\ \ \\ \ \\

\noindent \textbf{Consider $p=4k+1$:}
\\
Notice that the sum of $x$ and $y$ for each solution is $k-m$ or $p+k-m$
\\
Suppose that $y=k-m-x$.
\\
By the Quadratic Legendre sum theorem, the contribution of a prime $p|4(x+y)+1$ to the sum is $(p-1)\js{xy}$. Thus we have
\bea
(p-1)\cdot \sum_{x(p)} \js{x\cdot (k-m-x)}&\ = \ &(p-1)\cdot \sum_{x(p)} \js{-x^2+(k-m)x)}\nonumber\\
&\ = \ &(p-1)\cdot \left(-\js{-1}\right).
\eea
As $p=4k+1$, $\js{-1}=-1$, we find
\be
(p-1)\cdot \sum_{x(p)} \js{x\cdot (k-m-x)}=p-1.
\ee
There is always one time such that $p\mid 4x+4y+4m+1$ and $x=y$, so there are $p-2$ times such that $x=y$, $p\not | 4x+4y+4m+1$.
\\
The contribution of the cases $x=y$ to the sum is
\bea
(p-2)\cdot (p-1) \js{xy}&\ = \ &(p-2)\cdot (p-1) \js{x^2}\nonumber\\
&\ = \ &(p-2)(p-1).
\eea
We know that $\sum_{x(p)}\sum_{y_(p)} \js{xy}=0$.
\\
For all the cases such that $p\mid 4xy(x-y)(4x+4y+4m+1)$, the sum of $\js{xy}$ is $1+(p-2)=p-1$.
\\
Therefore, for all the cases such that $p\not | 4xy(x-y)(4x+4y+4m+1)$, the sum of $\js{xy}$ is $-(p-1)=1-p$.
\\
The contribution of all the cases such that $p\not | 4xy(x-y)(4x+4y+4m+1)$ to the second moment sum is $-\js{xy}$, which is $p-1$.
\ \\ \ \\
For $p=4k+1$,
\be
A_{2,E}(p)= p-1+(p-2)\cdot (p-1)+p-1=p^2-p.
\ee
\ \ \\ \ \\ \ \\
\noindent \textbf{Consider now $p=4k+3$.}
\\
Notice that the sum of $x$ and $y$ for each solution is $3k+2-m$ or $p+3k+2-m$.
\\
Suppose that $y=3k+2-m-x$.
\\
By the Quadratic Legendre sum theorem, the contribution of a prime $p|4(x+y+m)+1$ to the sum is
\bea
(p-1)\js{xy}&\ = \ &(p-1)\cdot \sum_{x(p)} \js{x\cdot (3k+2-m-x)}\nonumber\\
&\ = \ &(p-1)\cdot \sum_{x(p)} \js{-x^2+(3k+2-m)x}\nonumber\\
&\ = \ &(p-1)\cdot \left(-\js{-1}\right).
\eea
As $p=4k+3$, we have $\js{-1}=1$.
\be (p-1)\cdot \sum_{x(p)} \js{x\cdot (3k+2-m-x)}=1-p.\ee
\\
There is always one time such that $p\mid 4x+4y+4m+1$ and $x=y$, so there are $p-2$ times such that $x=y$, $p\not | 4x+4y+4m+1$.
\\
The contribution of the cases $x=y$ to the sum is
\bea
(p-2)\cdot (p-1) \js{xy}&\ = \ &(p-2)\cdot (p-1) \js{x^2}\nonumber\\
&\ = \ &(p-2)(p-1).
\eea
We know that $\sum_{x(p)}\sum_{y_(p)} \js{xy}=0$.
\\
For all the cases such that $p\mid 4xy(x-y)(4x+4y+4m+1)$, the sum of $\js{xy}$ is $-1+(p-2)=p-3$.
\\
Therefore, for all the cases such that $p\not | 4xy(x-y)(4x+4y+4m+1)$, the sum of $\js{xy}$ is $-(p-3)=3-p$.
\\
The contribution of all the cases such that $p\not | 4xy(x-y)(4x+4y+4m+1)$ to the second moment sum is the sum of $-\js{xy}$, which is $p-3$.
\ \\ \ \\
For $p=4k+1$,
\be A_{2,E}(p)= 1-p+(p-2)\cdot (p-1)+p-3=p^2-3p.\ee
\\

Therefore, the above analysis has shown that
\begin{equation}
A_{2,E}(p)\ =\
\Bigg\{ {\ p^2-3p \ \
\mbox{if} \  p=4k+1 \atop p^2-p \ \mbox{if} \ p=4k+3.}
\end{equation}
\end{proof}

Three of the families that we investigated through numerical computations are special cases of this elliptic curves family.

\subsubsection{$y^2+xy=x^3+tx$}
($a_1=1$ $a_2=0$ $a_3=0$ $a_4=t$ $a_6=0$).
\be
y^2\ = \ 4x^3+x^2+4tx.
\ee
\begin{equation}
A_{2,E}(p) =
\Bigg\{ {\ p^2-3p \ \
\mbox{if} \  p=4k+1 \atop p^2-p \ \mbox{if} \ p=4k+3}
\end{equation}
\subsubsection{$y^2+xy=x^3-2x^2+tx$}
($a_1=1$ $a_2=-2$ $a_3=0$ $a_4=t$ $a_6=0$).
We first put in Weierstrass form:
\be
y^2 \ = \ 4x^3-7x^2+4tx.
\ee
We have
\begin{equation}
A_{2,E}(p)\ =\
\Bigg\{ {\ p^2-3p \ \
\mbox{if} \  p=4k+1 \atop p^2-p \ \mbox{if} \ p=4k+3.}
\end{equation}

\subsubsection{$y^2+xy=x^3+x^2+tx$}
($a_1=1$ $a_2=1$ $a_3=0$ $a_4=t$ $a_6=0$).
Changing variables to have it in Weierstrass form gives
\be
y^2 \ = \ 4x^3+5x^2+4tx.
\ee
We find
\begin{equation}
A_{2,E}(p)\ =\
\Bigg\{ {\ p^2-3p \ \
\mbox{if} \  p=4k+1 \atop p^2-p \ \mbox{if} \ p=4k+3.}
\end{equation}

\subsection{First and second moments of the family$y^2=x^3-t^2x+t^4$}
($a_1=0$ $a_2=0$ $a_3=0$ $a_4=-t^2$ $a_6=t^4$)
\begin{lemma}The first moment of the family $y^2=x^3-t^2x+t^4$ is $-2p$.
\end{lemma}

\begin{proof} We have
\bea
A_{1,\varepsilon(p)}&\ = \ & -\sum_{t(p)} \sum_{x(p)} \js{x^3-t^2x+t^4}\nonumber
\nonumber\\
&\ = \ &-\sum_{t(p)} \js{t^4}-\sum_{t(p)} \sum_{x=1}^{p-1} \js{x^3-t^2x+t^4}\nonumber
\nonumber\\
&\ = \ &-p+1-\sum_{t(p)} \sum_{x=1}^{p-1} \js{t^3x^3-t^3x+t^4}\nonumber
\nonumber\\
&\ = \ &-p+1-\sum_{t(p)} \sum_{x=1}^{p-1} \js{t^2} \js{t^2+t(x^3-x)}\nonumber
\nonumber\\
&\ = \ &-p+1-\sum_{x=1}^{p-1} \sum_{t(p)} \js{t^2+t(x^3-x)}.
\eea
We compute the discriminant of the polynomial in $t$:
\be
\Delta \ = \ (x^3-x)^2=[(x-1)x(x+1)]^2
\ee

Note $p\mid \Delta$ when $x=1$ or $p-1$, and $p\not | \Delta$ when $x=2, 3, \dots, p-2$.

By the Quadratic Legendre Sum Theorem,
\be
A_{1,\varepsilon(p)} \ = \ -p+1-2(p-1)+(p-3) \ = \ -2p.
\ee
\end{proof}

\begin{lemma} The second moment of the family $y^2=x^3-t^2x+t^4$ is
\begin{equation}
A_{2,E}(p) \ = \ p^2-p-p\cdot \js{-3}-p\cdot \js{12}-\sum_{x(p)} \sum_{y(p)} \js{x^3-x} \js{y^3-y}
\end{equation}
\end{lemma}

\begin{proof} We have
\bea
A_{2,E}(p)&\ = \ & \sum_{t(p)} \sum_{x(p)} \sum_{y(p)} \js{x^3-t^2x+t^4} \js{y^3-t^2y+t^4}
\nonumber\\
&\ = \ &\sum_{t(p)} \sum_{x(p)} \sum_{y(p)} \js{t^3x^3-t^3x+t^4} \js{t^3y^3-t^3y+t^4}\nonumber\\
&\ = \ &\sum_{t(p)} \sum_{x(p)} \sum_{y(p)} \js{t^6} \js{x^3-x+t} \js{y^3-y+t}\nonumber\\
&\ = \ &\sum_{t(p)} \sum_{x(p)} \sum_{y(p)} \js{t^2+(x^3-x+y^3-y)t+(x^3-x)(y^3-y)}-\sum_{x(p)} \sum_{y(p)} \js{x^3-x} \js{y^3-y}\nonumber\\
a(x)&\ = \ &x^3-x \nonumber\\
b(y)&\ = \ &y^3-y \nonumber\\
A_{2,E}(p) &=& \sum_{t(p)} \sum_{x(p)} \sum_{y(p)}  \js{t^2+(a+b)t+ab}-\sum_{x(p)} \sum_{y(p)} \js{x^3-x} \js{y^3-y}.
\eea
The discriminant of $t^2+(a+b)t+ab$ is
\be
\Delta_t=(a+b)^2-4ab=(a-b)^2
\ee
\bea
\delta^2&\ = \ &\Delta_t\nonumber
\\
\delta&\ = \ &(a-b)\nonumber
\\
&\ = \ &x^3-x-y^3+y\nonumber
\\
&\ = \ &(x-y)(x^2+xy+y^2-1)
\eea
If $p|\delta$, then $p|x-y$ or $p|x^2+xy+y^2-1$
\\
x, y range from 0 to p-1, so $p|x-y$ p times.
\\
When $p|x^2+xy+y^2-1$,
\\
By quadratic formula mod $p$,
\\
$y^2+xy+x^2-1\equiv 0 \pmod p$ when
\\
$y=\frac{-x\pm \sqrt{\Delta_y}}{2}$  ($\Delta_y$ is the discriminant of the polynomial $y^2+xy+x^2-1$ in terms of y)
\bea
\Delta_y&\ = \ &x^2-4\cdot (x^2-1)\nonumber
\\
&\ = \ &-3x^2+4
\eea
\\
If $\Delta_y$ is a non-zero square mod $p$, there are two solutions. If $\Delta_y$ is 0 mod $p$, there is one solution. If $\Delta_y$ is not a square mod $p$, there is no solution.
\\
The number of pairs of $x,y$ such that $p|y^2+xy+x^2-1$ is
\be
\sum_{x(p)} 1+\js{-3x^2+4}=p+\sum_{x(p)} \js{-3x^2+4}.
\ee
The discriminant of $-3x^2+4$ is
\bea
\Delta_x&\ = \ &0-4(-3)\cdot 4 \nonumber\\
&\ = \ &48
\eea
By the Quadratic Legendre Sum Theorem, if $p\not |2,3$,
\be
\js{-3x^2+4} \ = \ -\js{-3}
\ee
The number of pairs of $x,y$ such that $p|y^2+xy+x^2-1$ is
\be
p+\sum_{x(p)} \js{-3x^2+4}=p-\js{-3}.
\ee
We need the number of cases that we double count $x=y$ and $p|y^2+xy+x^2-1$.
If $x=y$,
\be
y^2+xy+x^2-1\equiv 3y^2-1\equiv0 \pmod p.
\ee
The discriminant of $3y^2-1$ is
\bea
\Delta_y&\ = \ &0-4\cdot 3\cdot (-1)\nonumber\\
&\ = \ &12.
\eea
By the quadratic Legendre sum formula mod $p$, the number of solutions is $1+\js{12}$.
Therefore, the total number of times that $p|(x-y)(y^2+xy+x^2-1)$ is the number of times $p|x-y$ plus the number of times $p|y^2+xy+x^2-1$ minus the cases that we double count.
\be
p+p-\js{-3}-1-\js{12} \ = \ 2p-1-\js{-3}-\js{12}.
\ee
The number of times that $p\not |(x-y)(4x^2+4xy+4y^2+x+y-4)$ is $p^2-(2p-1-\js{-3}-\js{12})$. Again by Quadratic Legendre Sum Theorem,
\bea
A_{2,E}(p)&\ = \ &(p-1)\left[2p-1-\js{-3}-\js{12}\right]-\left[p^2-(2p-1-\js{-3}-\js{12})\right]\nonumber\\ & & \ \ \ \ - \ \sum_{x(p)} \sum_{y(p)} \js{x^3-x} \js{y^3-y}\nonumber\\
&\ = \ &p^2-p-p\cdot \js{-3}-p\cdot \js{12}-\sum_{x(p)} \sum_{y(p)} \js{x^3-x} \js{y^3-y}.
\eea

Thus
\bea
\sum_{x(p)} \sum_{y(p)} \js{x^3-x} \js{y^3-y}
&\ = \ &\sum_{x(p)} \js{x^3-x} \sum_{y(p)} \js{y^3-y}\nonumber
\\
&\ = \ &\sum_{x(p)} \js{x^3-x} \sum_{y(p)} \js{(y-1)y(y+1)}.
\eea
We know that $\js{-1}=(-1)^{\frac{p-1}{2}}$.
\\
When $p=4k+3$, $\js{-1}=-1$
\bea
\sum_{y(p)} \js{(y-1)y(y+1)}&\ = \ &\sum_{y=2}^{2k} \js{(y-1)y(y+1)}+\sum_{y=2k+1}^{4k-1} \js{(y-1)y(y+1)}\nonumber\\
&\ = \ &\sum_{y=2}^{2k} \js{(y-1)y(y+1)}+\js{-1}\js{(y-1)y(y+1)}\nonumber\\
&\ = \ &0.
\eea
In conclusion,
\be
A_{2,E}(p) \ = \ p^2-p-p\cdot \js{-3}-p\cdot \js{12}-\sum_{x(p)} \sum_{y(p)} \js{x^3-x} \js{y^3-y}.
\ee
When $p\equiv 3\pmod{4}$,
\be
A_{2,E}(p) \ = \ p^2-p-p\cdot \js{-3}-p\cdot \js{12}.
\ee
\end{proof}

\subsection{Numerical data for second moment sums}
Now we report on families where we cannot find closed-form polynomials for their second moments sums. These are more generic families than the ones investigated both above and in previous work, and provide a new and stronger test of the bias conjecture.

For these families, we have calculated the second moment sums for the first $1000$ primes. By Michel's theorem, we know that the main term of the sum is $p^2$, and lower order terms have size $p^{3/2}, p, p^{1/2}$ or $1$. From the data we have, we can tell if it is likely that the second moment has a $p^{3/2}$ term. If the value of $\frac{\text{second\ moment}-p^2}{p}$ converges or stays bounded as the prime grows, then it is likely that the largest lower order term of the second moment sum is $p$, as if there were a $p^{3/2}$ term we would have fluctuations of size $p^{1/2}$.

By subtracting the main term ($p^2)$ from the sum and then dividing by the largest lower term ($p^{3/2}$ or $p$), we calculated the average bias; see Figure \ref{bias_2nd}.

\begin{figure}[ht]
\begin{center}
\scalebox{1.1}{\includegraphics{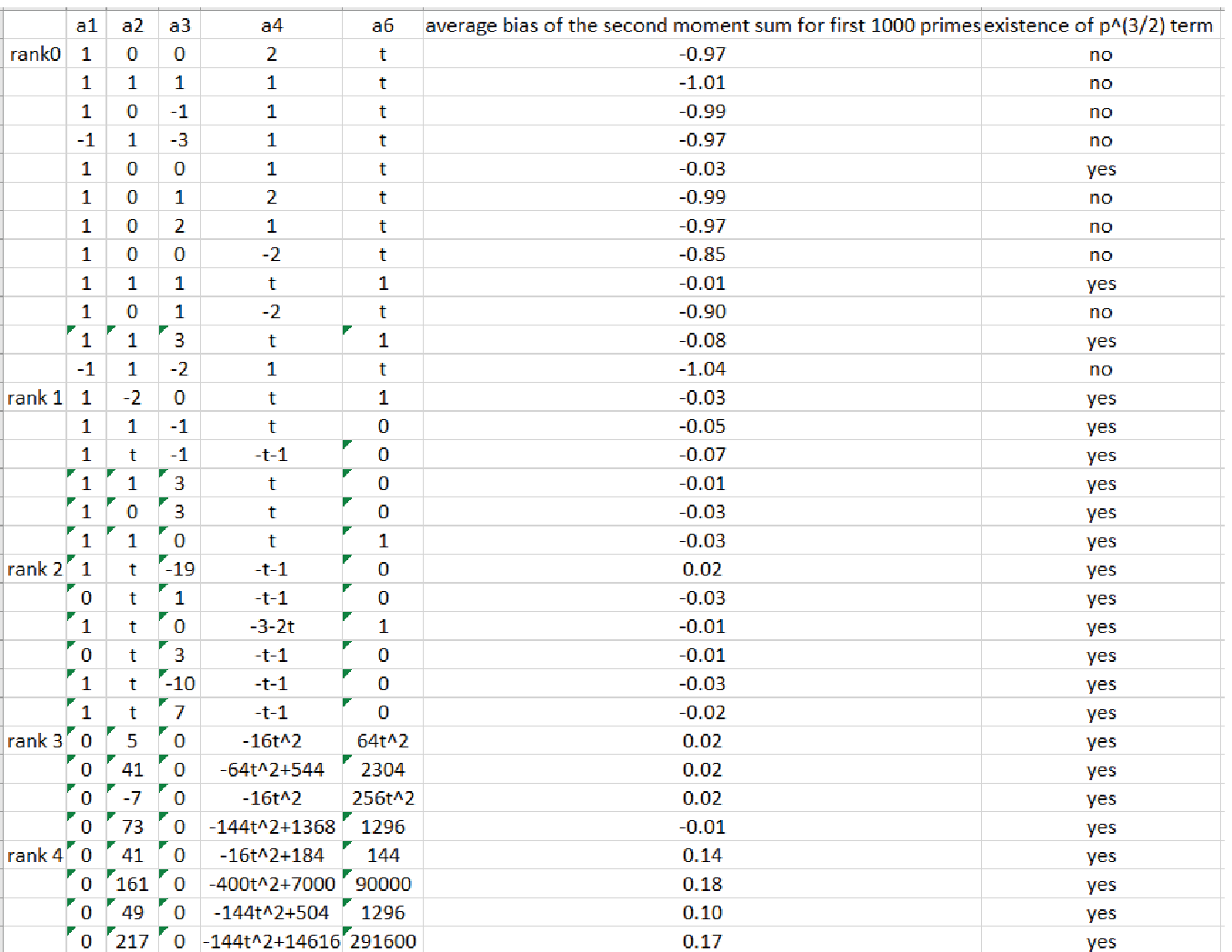}}
\caption{\label{bias_2nd} Numerical data for the average biases of second moments sums.}
\end{center}
\end{figure}

The data shows that all the families where we do not believe there is a $p^{3/2}$ term clearly have negative biases (around $-1$). When the $p^{3/2}$ exists, the bias unfortunately becomes impossible to see. The reason is that the $p^{3/2}$ term drowns it out; we now have to divide by $p^{3/2}$. If that term averages to zero, then the term of size $p$, once we divide by $p^{3/2}$, is of size $1/p^{1/2}$.

Let's investigate further the consequence of having a term of size $p^{3/2}$. We divide the difference of the observed second moment minus $p^2$ (the expected value) by $p^{3/2}$. We now have signed summands of size 1. By the Philosophy of Square-Root Cancellation, if we sum $N$ such signed terms we expect a sum of size $\sqrt{N}$. As we are computing the average of these second moments, we divide by $N$ and have an expected value of order $1/\sqrt{N}$. In other words, if the $p^{3/2}$ term is present and averages to zero, we expect sums over ranges of primes to be about $1/\sqrt{N}$. If $N=1000$ this means we expect sums on the order of .0316. Looking at the data in Figure \ref{bias_2nd}, what we see is consistent with this analysis. Thus, while we cannot determine if the first lower order term that does not average to zero has a negative bias, we can at least show that the data is consistent with the $p^{3/2}$ term averaging to zero for lower rank families.

The table suggests a lot more. From the data, we can see that all the rank $0$ and rank $1$ families have negative biases. However, all four rank $4$ families have shown positive biases from the first 1000 primes. Thus, we look further to see if it is likely the result of fluctuations, or if perhaps it is evidence against the bias conjecture.

We divide the $1000$ primes into $20$ groups of $50$ for further analysis. If the $p^{3/2}$ term averages to zero, we would expect each of these groups to be positive and negative equally likely, and we can compare counts. We now expect each group to be on the order of $1/\sqrt{50} \approx .14$. Thus we shouldn't be surprised if it is a few times .14 (positive or negative); remember we do not know the constant factor in the $p^{3/2}$ term and are just doing estimates.

For the rank $2$ family $a_1=1$, $a_2=t$, $a_3=-19$, $a_4=-t-1$, $a_6=0$, $12$ of the $20$ groups of primes have shown positive biases. Figure \ref{1t(-19)(-t-1)0} is a histogram plot of the distribution of the average biases among the 20 groups.

\begin{figure}[ht]
\begin{center}
\scalebox{1}{\includegraphics{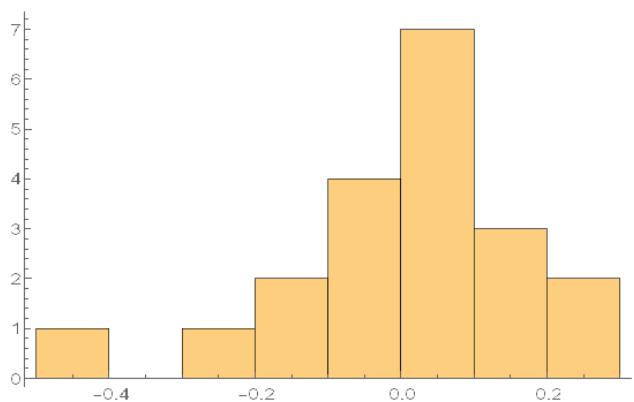}}
\caption{\label{1t(-19)(-t-1)0} Distribution of average biases in the first 1000 primes for family $a_1=1$, $a_2=t$, $a_3=-19$, $a_4=-t-1$, $a_6=0$.}
\end{center}
\end{figure}

We now further analyze our data by dividing the 1000 primes into 100 groups of 10 for this family. As shown from the data, $58$ of the $100$ groups of primes have shown positive biases; however, this is still consistent with a term that averages to zero by the Philosophy of Square-Root Cancelation or the Central Limit Theorem (if we have 100 outcomes that are positive half the time and negative half the time, the expected number of positive outcomes is 50 and the standard deviation is 5; thus having 58 positive and 42 negative groups is well-within two standard deviations). Figure \ref{1t(-19)(-t-1)0_10th} is a histogram plot of the distribution of the average biases among the $100$ groups. These numbers are consistent with the $p^{3/2}$ term averaging to zero.

\begin{figure}[ht]
\begin{center}
\scalebox{1}{\includegraphics{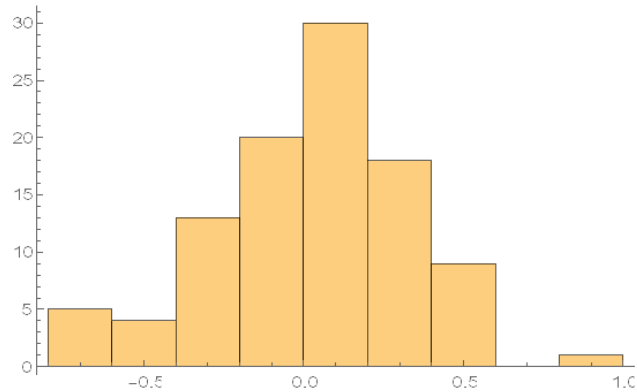}}
\caption{\label{1t(-19)(-t-1)0_10th} Distribution of average biases in the first 1000 primes for family $a_1=1$, $a_2=t$, $a_3=-19$, $a_4=-t-1$, $a_6=0$.}
\end{center}
\end{figure}

For the rank $3$ family $a_1=0$, $a_2=5$, $a_3=0$, $a_4=-16t^2$, $a_6=64t^2$, $12$ of the $20$ groups of primes have shown positive biases. Figure \ref{050-16t_264t_2A} is a histogram plot of the distribution of the average biases among the 20 groups. Again the results are consistent with the $p^{3/2}$ term averaging to zero.

\begin{figure}[ht]
\begin{center}
\scalebox{1}{\includegraphics{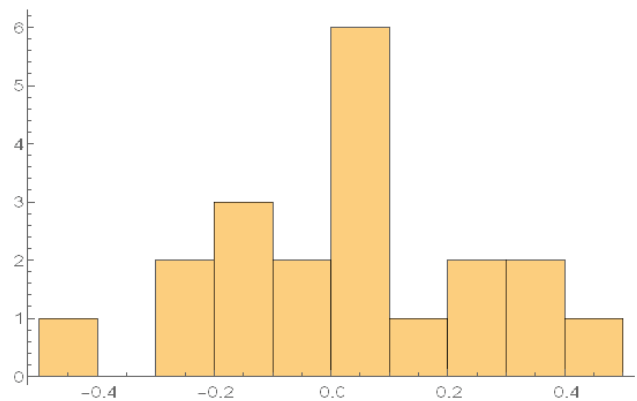}}
\caption{\label{050-16t_264t_2A} Distribution of average biases in the first 1000 primes for family $a_1=0$, $a_2=5$, $a_3=0$, $a_4=-16t^2$, $a_6=64t^2$.}
\end{center}
\end{figure}

From the primes that we have analyzed, the two families above appear to have positive biases more frequently than negative biases, \emph{but not by a statistically significant margin}. Now we compare them to some families that have shown negative biases in the first 1000 primes.
\ \\ \ \\
For the rank $0$ family $a_1=1$, $a_2=0$, $a_3=0$, $a_4=1$, $a_6=t$, \emph{\textbf{all}} of the $20$ groups of primes have shown negative biases. Figure \ref{1001t} is a histogram plot of the distribution of the average biases among the 20 groups.

\begin{figure}[ht]
\begin{center}
\scalebox{1}{\includegraphics{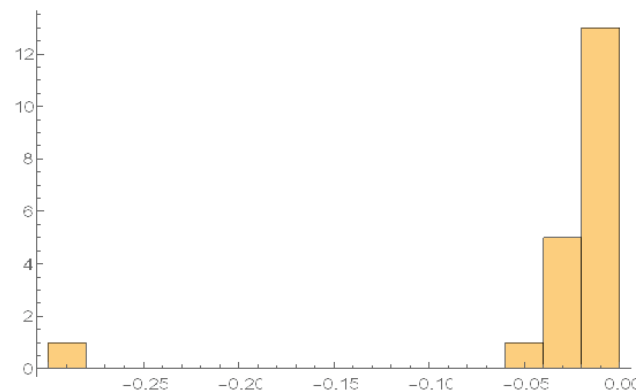}}
\caption{\label{1001t}  Distribution of average biases in the first 1000 primes for family $a_1=1$, $a_2=0$, $a_3=0$, $a_4=1$, $a_6=t$.}
\end{center}
\end{figure}

We now further analyze our data by dividing the 1000 primes into 100 groups of 10 for this family. As shown by the data, \emph{\textbf{all}} of the $100$ groups of primes have shown negative biases, which strongly indicates that the negative bias exists in this family. Figure \ref{1001t_10primes} is a histogram plot of the distribution of the average biases among the $100$ groups.

\begin{figure}[ht]
\begin{center}
\scalebox{0.95}{\includegraphics{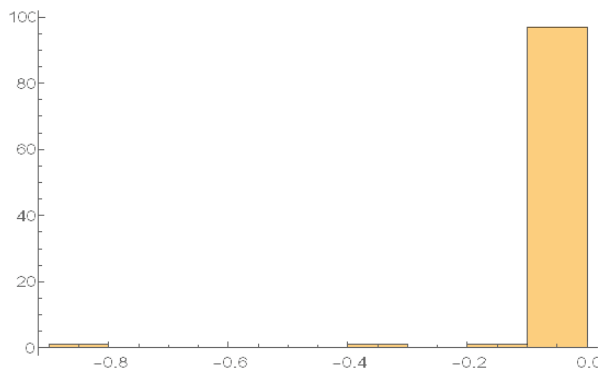}}
\caption{\label{1001t_10primes} Distribution of average biases in the first 1000 primes for family $a_1=1$, $a_2=0$, $a_3=0$, $a_4=1$, $a_6=t$}
\end{center}
\end{figure}

For the rank $1$ family $a_1=1$, $a_2=1$, $a_3=-1$, $a_4=t$, $a_6=0$, $13$ of the $20$ groups of primes have shown negative biases. Figure \ref{11-1t0} is a histogram plot of the distribution of the average biases among the 20 groups.

\begin{figure}[ht]
\begin{center}
\scalebox{0.9}{\includegraphics{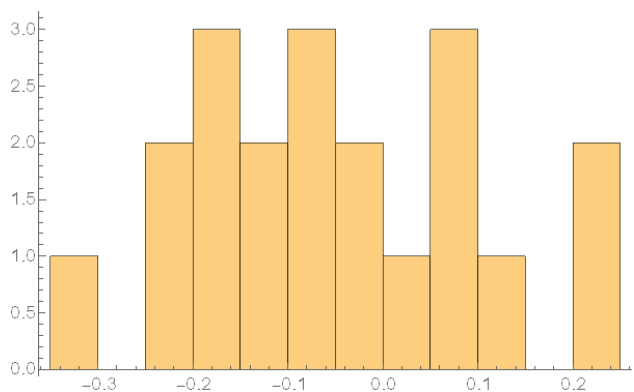}}
\caption{\label{11-1t0} Distribution of average biases in the first 1000 primes for family $a_1=1$, $a_2=1$, $a_3=-1$, $a_4=t$, $a_6=0$.}
\end{center}
\end{figure}

For the rank $1$ family $a_1=1$, $a_2=t$, $a_3=-1$, $a_4=-t-1$, $a_6=0$, $13$ of the $20$ groups of primes have shown negative biases. Figure \ref{1t-1-t-10} is a histogram plot of the distribution of the average biases among the 20 groups.

\begin{figure}[ht]
\begin{center}
\scalebox{0.9}{\includegraphics{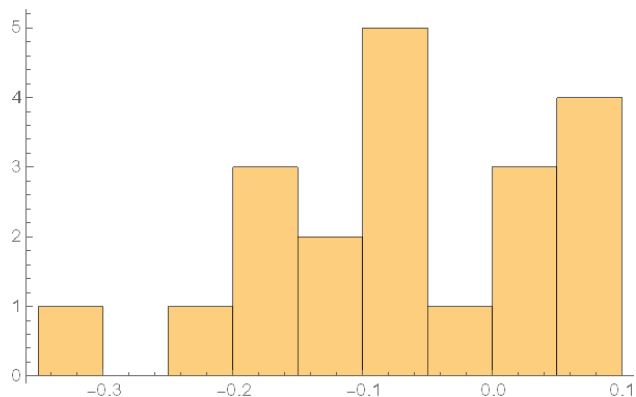}}
\caption{\label{1t-1-t-10} Distribution of average biases in the first 1000 primes for family $a_1=1$, $a_2=t$, $a_3=-1$, $a_4=-t-1$, $a_6=0$}
\end{center}
\end{figure}

From these data, the rank $0$ and rank $1$ families have negative biases more frequently, but we are working with small data sets and must be careful in how much weight we assign such results. In the rank $0$ family, \emph{\textbf{all}} the groups of primes have negative biases. In both of the rank $1$ families, $65\%$ of the groups of the primes have negative biases, though this percentage is not statistically significant. On the other hand, two of the three rank $2$ families appear to have positive biases in the first $1000$ primes, but again this value is not statistically significant. We believe that the rank of the families might play a role in determining the bias. Therefore, it is possible that the negative bias conjecture does not hold for some families with larger rank.

To further examine the biases in families with larger ranks, we investigate the rank 6 family $a_1=0$, $a_2=2(16660111104 t) + 811365140824616222208$, $a_3=0$, $a_4=[2(-1603174809600)t-26497490347321493520384](t^2+2t-8916100448256000000+1)$, $a_6=[2(2149908480000)t+343107594345448813363200](t^2+2t-8916100448256000000+1)^2$. Our data suggests that there is a positive bias in this family. The average bias of second moments sums for the first $1000$ primes is $0.246759$. Figure \ref{811365140824616222208_2nd} is a histogram plot of the distribution of the average biases among the $100$ groups of $10$ primes. $78$ of the $100$ groups of primes have positive biases, which suggests that it is likely that the second moment of this family has a positive $p^{3/2}$ term. 

\begin{figure}[ht]
\begin{center}
\scalebox{0.9}{\includegraphics{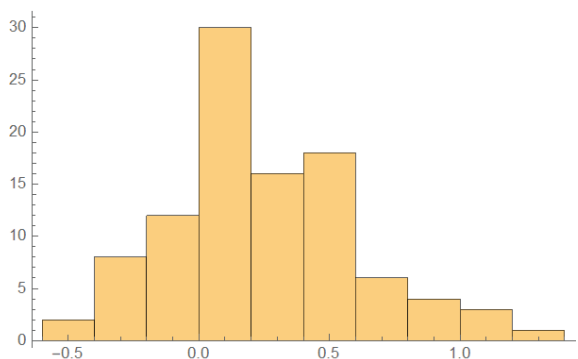}}
\caption{\label{811365140824616222208_2nd} Distribution of average biases in the first 1000 primes for a rank 6 family.}
\end{center}
\end{figure}

Therefore, our data for the rank $4$ and rank $6$ families has shown that it is likely that higher rank families (${\rm rank}(E(\Q))\ge4$) have positive biases.

\section{Biases in fourth and six moments of elliptic curve families}\label{sec:three}
We now explore, for the first time, the higher moments of the Dirichlet coefficients of the elliptic curve $L$-functions to see if biases we found in the first and second moments persist. Unfortunately existing techniques on analyzing the second moment sums do not apply to the higher moments, even if we choose nice families. If we switch orders of the moments' sums and sum over $t$, we are going to get a cubic or higher degrees polynomials. Therefore, we can only try to predict or observe the biases through numerical evidence. We calculated the 4th and 6th moment sums for the first $1000$ primes. From Section \ref{4th_moment_form}, we know that the main term of the fourth moment sum is $2p^3$, and the largest possible lower order terms have size $p^{5/2}$. From Section \ref{6th_moment_form}, we know that the main term of the sixth moment sum is $5p^4$, and the largest possible lower order terms have size $p^{7/2}$. From the data we have gathered, all the 4th moments of these families have $p^{5/2}$ terms, and all the 6th moments have $p^{7/2}$ terms. By subtracting the main term ($2p^3)$ from the fourth moment sum and then dividing by the size of the largest lower term ($p^{5/2}$), we calculated the average bias for the fourth moment of the first $1000$ primes. Similarly, we subtracted $5p^4$ from the sixth moment sum and then divided by $p^{7/2}$ to calculate the average bias for the sixth moment of the first $1000$ primes; See Figure \ref{4_6moments}.

\begin{figure}[ht]
\begin{center}
\scalebox{1}{\includegraphics{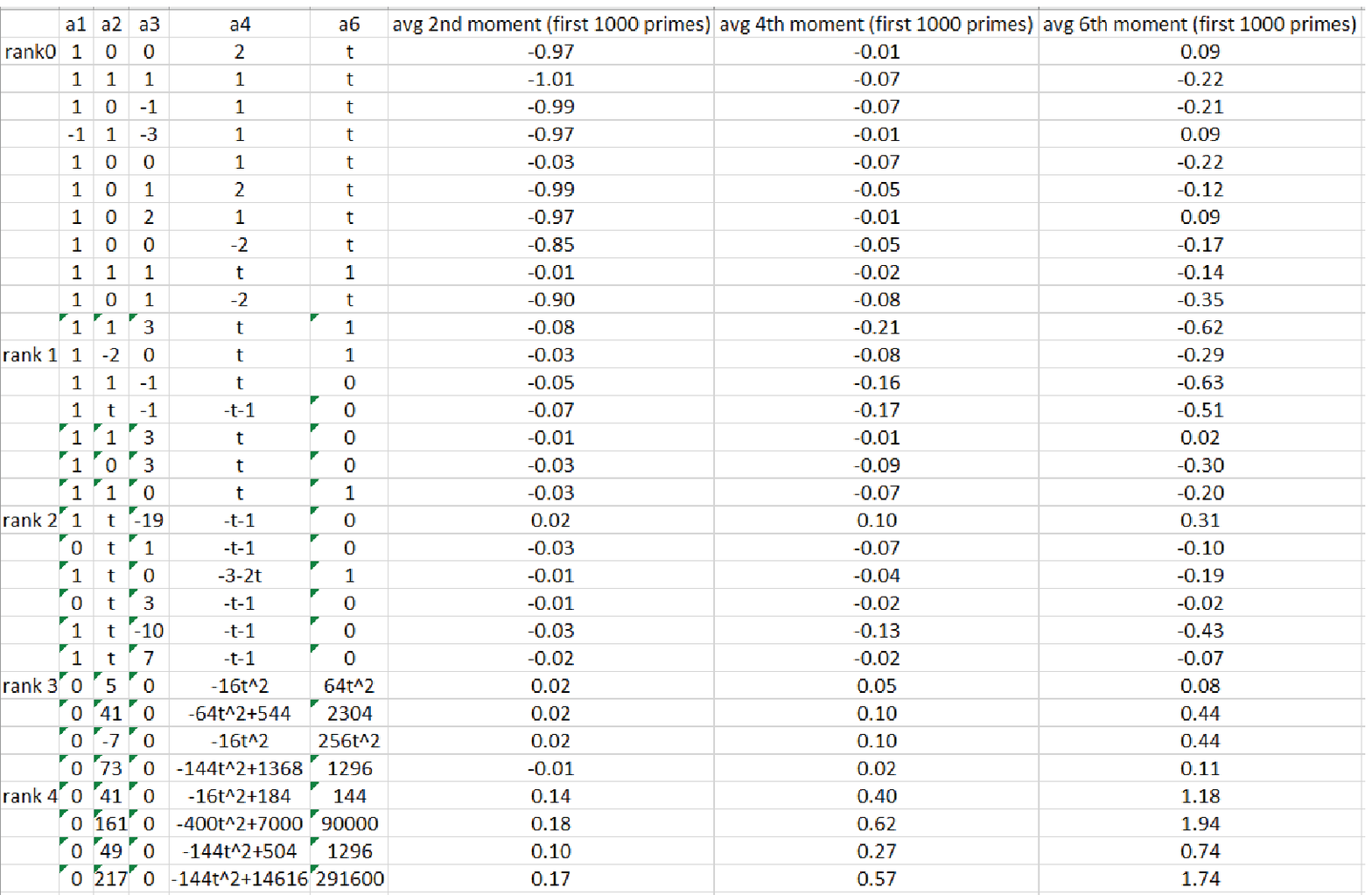}}
\caption{\label{4_6moments}Numerical data for the average biases of 2nd, 4th and 6th moments sums.}
\end{center}
\end{figure}
\subsection{Biases in fourth moment sums}
From the data, we can see that all the biases for lower rank families in the fourth moment are relatively small (smaller than $0.2$), which indicates that the $p^{5/2}$ term likely averages to $0$. By the Philosophy of Square-Root Cancellation, we expect the order of the size of fluctuation to be around $\sqrt{1000}/1000 \approx 0.03$. Therefore, if the bias is between $-0.2$ and $0.2$, we would expect $p^2$ to be the largest lower order term.

Note that for $30$ out of $31$ families, the bias in fourth moments appear to be similar to the bias in second moments (families that have negative bias in second moments also seem to have negative bias in fourth moments, and vice versa), though much smaller magnitudes likely due to the presence of a $p^{5/2}$ term that is averaging to zero. We now explore a few families whose 4-th moment biases have different scales in magnitudes.

For the rank $0$ family $a_1=1$, $a_2=0$, $a_3=0$, $a_4=2$, $a_6=t$, $10$ of the $20$ groups of primes have shown negative biases. Figure \ref{1002t_4th} is a histogram plot of the distribution of the average biases among the 20 groups.

\begin{figure}[ht]
\begin{center}
\scalebox{0.8}{\includegraphics{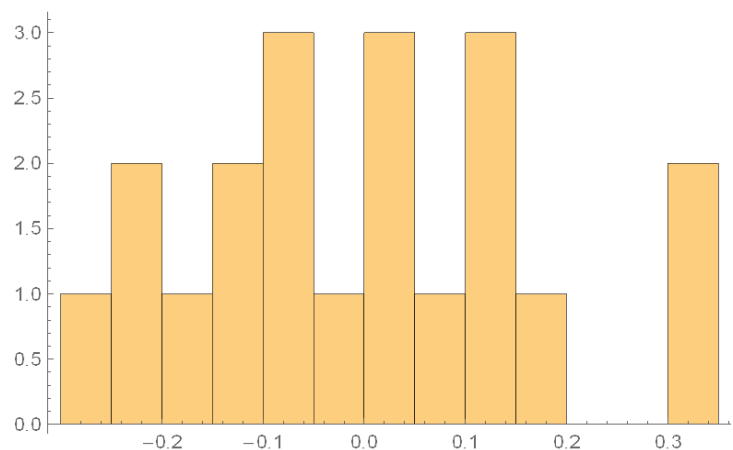}}
\caption{\label{1002t_4th} Distribution of average biases in the first 1000 primes for family $a_1=1$, $a_2=0$, $a_3=0$, $a_4=2$, $a_6=t$.}
\end{center}
\end{figure}

For the rank $0$ family $a_1=1$, $a_2=1$, $a_3=1$, $a_4=1$, $a_6=t$, $13$ of the $20$ groups of primes have shown negative biases. Figure \ref{1111t_4th} is a histogram plot of the distribution of the average biases among the 20 groups.
\begin{figure}[ht]
\begin{center}
\scalebox{0.7}{\includegraphics{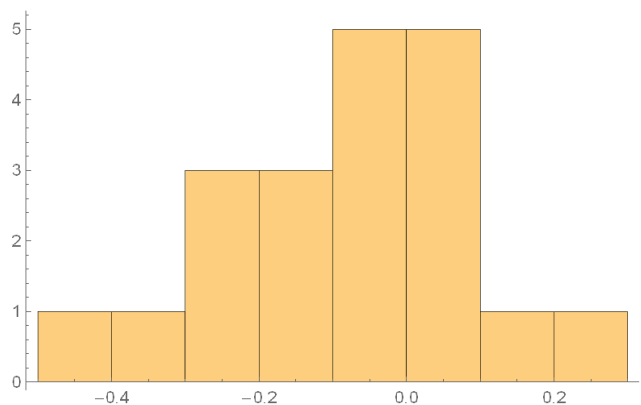}}
\caption{\label{1111t_4th} Distribution of average biases in the first 1000 primes for family $a_1=1$, $a_2=1$, $a_3=1$, $a_4=1$, $a_6=t$.}
\end{center}
\end{figure}

For the rank $1$ family $a_1=1$, $a_2=t$, $a_3=-1$, $a_4=-t-1$, $a_6=0$, $15$ of the $20$ groups of primes have shown negative biases. Figure \ref{1t-1-t-10_4th} is a histogram plot of the distribution of the average biases among the 20 groups.
\begin{figure}[ht]
\begin{center}
\scalebox{0.7}{\includegraphics{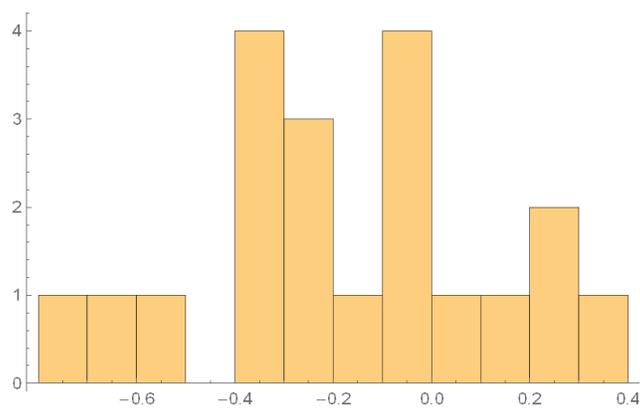}}
\caption{\label{1t-1-t-10_4th} Distribution of average biases in the first 1000 primes for family $a_1=1$, $a_2=t$, $a_3=-1$, $a_4=-t-1$, $a_6=0$.}
\end{center}
\end{figure}

We further analyze our data by dividing the 1000 primes into 100 groups of 10 for this family. As shown in Figure \ref{1t-1-t-10_4th_10primes}, $63$ of the $100$ groups of primes have shown negative biases. The probability of having 17 or more negatives than positives (or 17 or more positives than negatives) in 100 tosses of a fair coin (so heads is positive and tails is negative) is about 1.2\%. While unlikely, this is not exceptionally unlikely.

\begin{figure}[ht]
\begin{center}
\scalebox{1}{\includegraphics{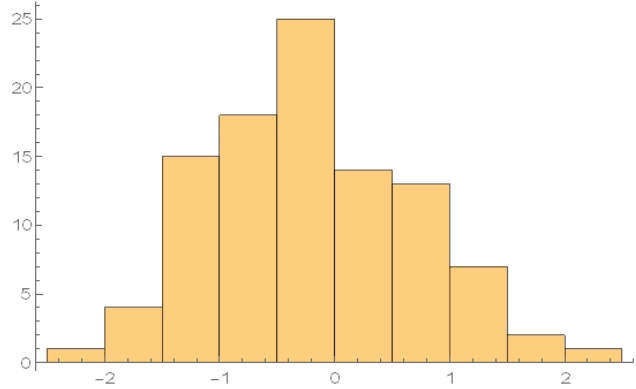}}
\caption{\label{1t-1-t-10_4th_10primes} Distribution of average biases in the first 1000 primes for family $a_1=1$, $a_2=t$, $a_3=-1$, $a_4=-t-1$, $a_6=0$.}
\end{center}
\end{figure}

For the rank $2$ family $a_1=1$, $a_2=t$, $a_3=-19$, $a_4=-t-1$, $a_6=0$, $7$ of the $20$ groups of primes have shown negative biases. Figure \ref{1t-19-t-10_4th} is a histogram plot of the distribution of the average biases among the 20 groups.
\begin{figure}[ht]
\begin{center}
\scalebox{0.7}{\includegraphics{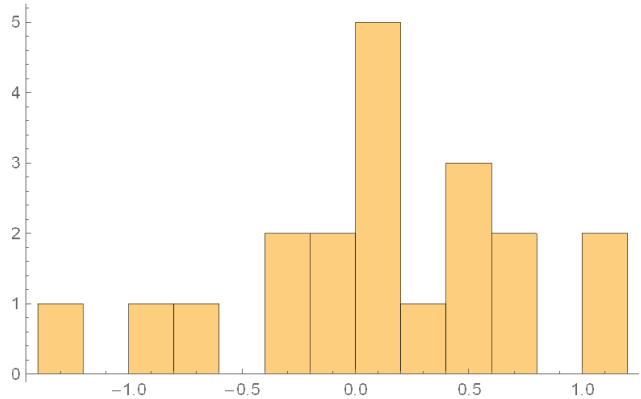}}
\caption{\label{1t-19-t-10_4th} Distribution of average biases in the first 1000 primes for family $a_1=1$, $a_2=t$, $a_3=-19$, $a_4=-t-1$, $a_6=0$.}
\end{center}
\end{figure}

We further analyzes our data by dividing the 1000 primes into 100 groups of 10 for this family. As shown in Figure \ref{1t-19-t-10_4th_10primes}, $44$ of the $100$ groups of primes have shown negative biases.
\begin{figure}[ht]
\begin{center}
\scalebox{1}{\includegraphics{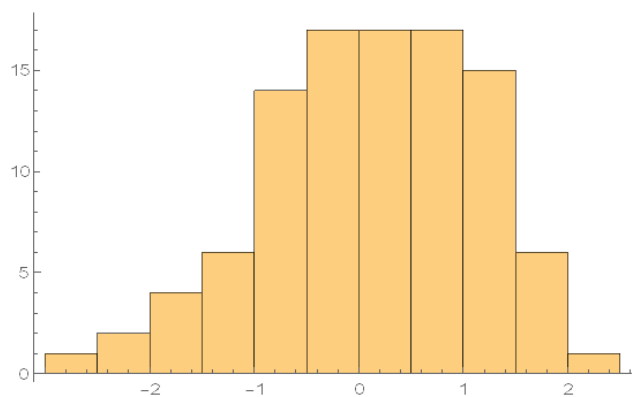}}
\caption{\label{1t-19-t-10_4th_10primes} Distribution of average biases in the first 1000 primes for family $a_1=1$, $a_2=t$, $a_3=-19$, $a_4=-t-1$, $a_6=0$.}
\end{center}
\end{figure}

For the rank $3$ family $a_1=0$, $a_2=5$, $a_3=0$, $a_4=-16t^2$, $a_6=64t^2$, $11$ of the $20$ groups of primes have shown negative biases. Figure \ref{050-16t_264t_2_4th} is a histogram plot of the distribution of the average biases among the 20 groups.
\begin{figure}[ht]
\begin{center}
\scalebox{0.75}{\includegraphics{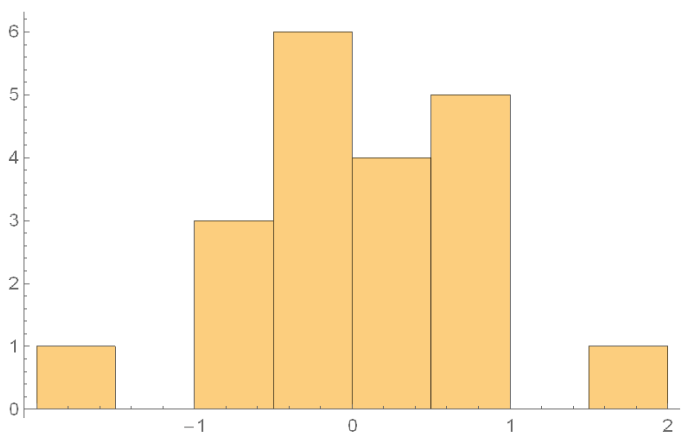}}
\caption{\label{050-16t_264t_2_4th} Distribution of average biases in the first 1000 primes for family $a_1=0$, $a_2=5$, $a_3=0$, $a_4=-16t^2$, $a_6=64t^2$.}
\end{center}
\end{figure}

Despite the fluctuations, all the rank $0$ and rank $1$ families seem to have negative biases more frequently in the first $1000$ primes, which suggests that it is possible that negative bias exists in the fourth moments of all rank $0$ and rank $1$ families.
Similar to the second moment sums, the fourth moment sums of families with larger rank appear to have positive biases for the first 1000 primes, but this might due to the fluctuations of the $p^{5/2}$ term as we are working with small data set.

To further examine the biases in families with larger ranks, we investigate the rank 6 family $a_1=0$, $a_2=2(16660111104 t) + 811365140824616222208$, $a_3=0$, $a_4=[2(-1603174809600)t-26497490347321493520384](t^2+2t-8916100448256000000+1)$, $a_6=[2(2149908480000)t+343107594345448813363200](t^2+2t-8916100448256000000+1)^2$. 
Our data suggests that there is a positive bias in this family. The average bias of the fourth moments sums for the first $1000$ primes is $0.753285$. Figure \ref{811365140824616222208_4th} is a histogram plot of the distribution of the average biases among the $100$ groups of $10$ primes. $75$ of the $100$ groups of primes have positive biases, which suggests that it is likely that the fourth moment of this family has a positive $p^{5/2}$ term. 

\begin{figure}[ht]
\begin{center}
\scalebox{0.9}{\includegraphics{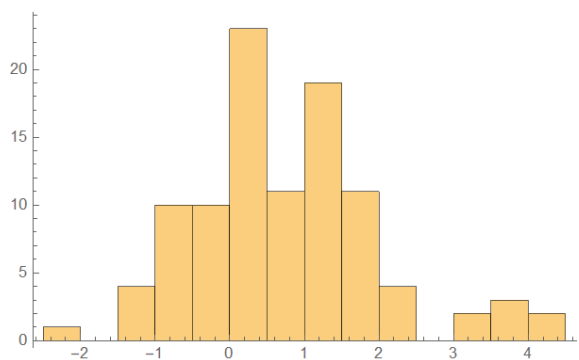}}
\caption{\label{811365140824616222208_4th} Distribution of average biases in the first 1000 primes for a rank 6 family.}
\end{center}
\end{figure}

Similar to the second moment, our data for the rank $4$ and rank $6$ families has shown that it is likely that higher rank families (${\rm rank}(E(\Q))\ge4$) have positive biases.

\subsection{Biases in sixth moment sums}
We now explore the 6th moment biases for these families.

For the rank $0$ family $a_1=1$, $a_2=0$, $a_3=0$, $a_4=2$, $a_6=t$, $10$ of the $20$ groups of primes have shown negative biases. Figure \ref{1002t_6th} is a histogram plot of the distribution of the average biases among the 20 groups.
\begin{figure}[ht]
\begin{center}
\scalebox{0.75}{\includegraphics{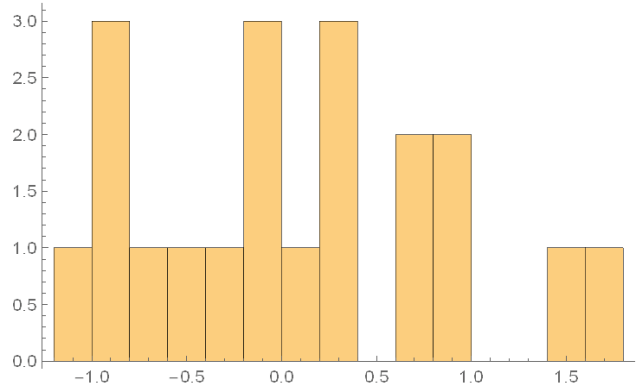}}
\caption{\label{1002t_6th} Distribution of average biases in the first 1000 primes for family $a_1=1$, $a_2=0$, $a_3=0$, $a_4=2$, $a_6=t$.}
\end{center}
\end{figure}

For the rank $0$ family $a_1=1$, $a_2=1$, $a_3=1$, $a_4=1$, $a_6=t$, $13$ of the $20$ groups of primes have shown negative biases. Figure \ref{1111t_6th} is a histogram plot of the distribution of the average biases among the 20 groups.
\begin{figure}[ht]
\begin{center}
\scalebox{0.75}{\includegraphics{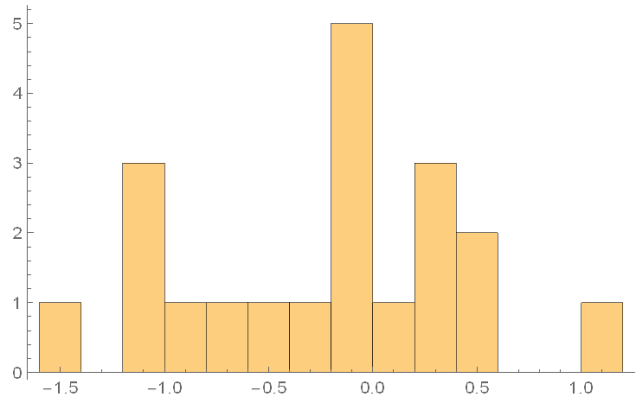}}
\caption{\label{1111t_6th} Distribution of average biases in the first 1000 primes for family $a_1=1$, $a_2=1$, $a_3=1$, $a_4=1$, $a_6=t$.}
\end{center}
\end{figure}

For the rank $1$ family $a_1=1$, $a_2=t$, $a_3=-1$, $a_4=-t-1$, $a_6=0$, $14$ of the $20$ groups of primes have shown negative biases. Figure \ref{1t-1-t-10_6th} is a histogram plot of the distribution of the average biases among the 20 groups.
\begin{figure}[ht]
\begin{center}
\scalebox{0.75}{\includegraphics{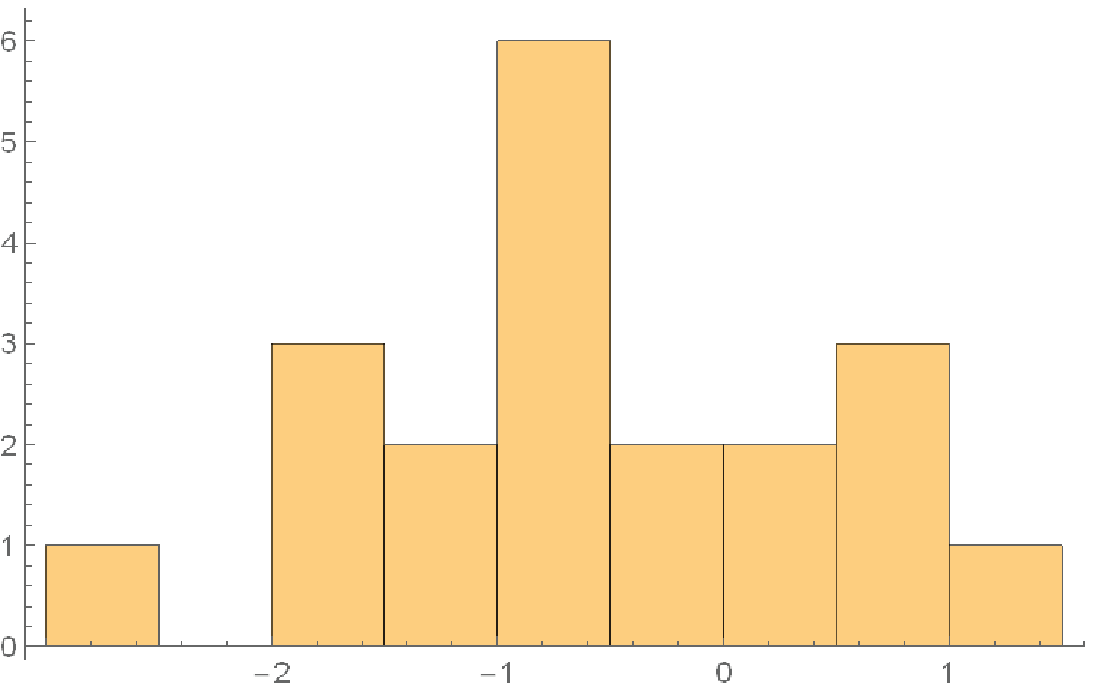}}
\caption{\label{1t-1-t-10_6th} Distribution of average biases in the first 1000 primes for family $a_1=1$, $a_2=t$, $a_3=-1$, $a_4=-t-1$, $a_6=0$.}
\end{center}
\end{figure}

We further analyze our data by dividing the 1000 primes into 100 groups of 10 for this family. As shown in Figure \ref{1t-1-t-10_6th_10primes}, $59$ of the $100$ groups of primes have shown negative biases.

\begin{figure}[ht]
\begin{center}
\scalebox{1}{\includegraphics{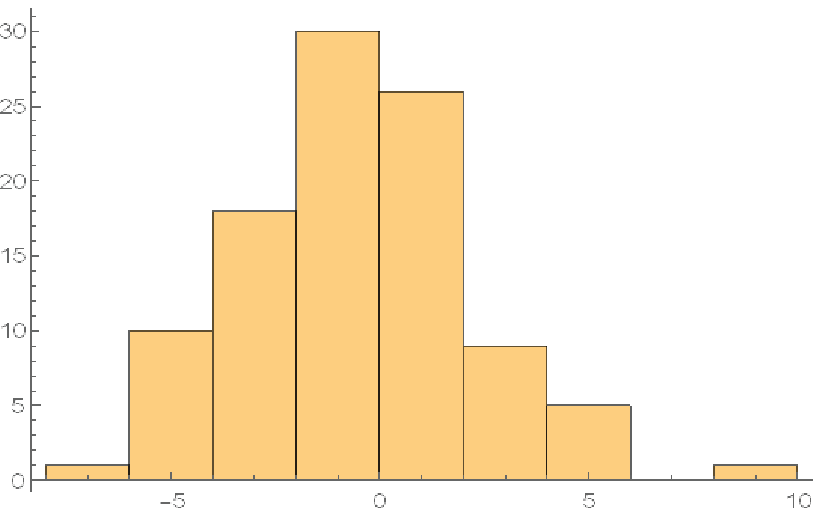}}
\caption{\label{1t-1-t-10_6th_10primes} Distribution of average biases in the first 1000 primes for family $a_1=1$, $a_2=t$, $a_3=-1$, $a_4=-t-1$, $a_6=0$.}
\end{center}
\end{figure}

For the rank $2$ family $a_1=1$, $a_2=t$, $a_3=-19$, $a_4=-t-1$, $a_6=0$, $10$ of the $20$ groups of primes have shown negative biases. Figure \ref{1t-19-t-10_6th} is a histogram plot of the distribution of the average biases among the 20 groups.
\begin{figure}[ht]
\begin{center}
\scalebox{0.75}{\includegraphics{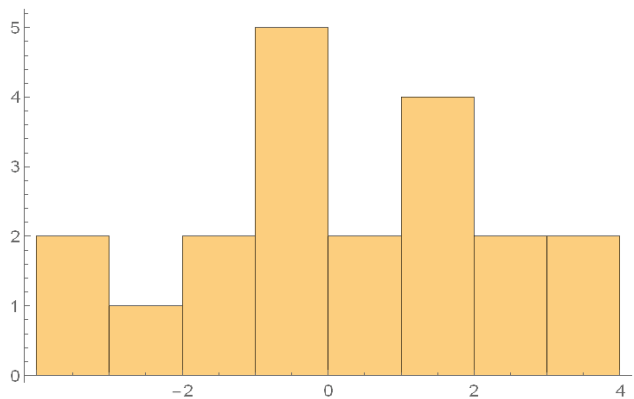}}
\caption{\label{1t-19-t-10_6th} Distribution of average biases in the first 1000 primes for family $a_1=1$, $a_2=t$, $a_3=-19$, $a_4=-t-1$, $a_6=0$.}
\end{center}
\end{figure}

For the rank $3$ family $a_1=0$, $a_2=5$, $a_3=0$, $a_4=-16t^2$, $a_6=64t^2$, $10$ of the $20$ groups of primes have shown negative biases. Figure \ref{050-16t_264t_2B} is a histogram plot of the distribution of the average biases among the 20 groups.
\begin{figure}[ht]
\begin{center}
\scalebox{0.75}{\includegraphics{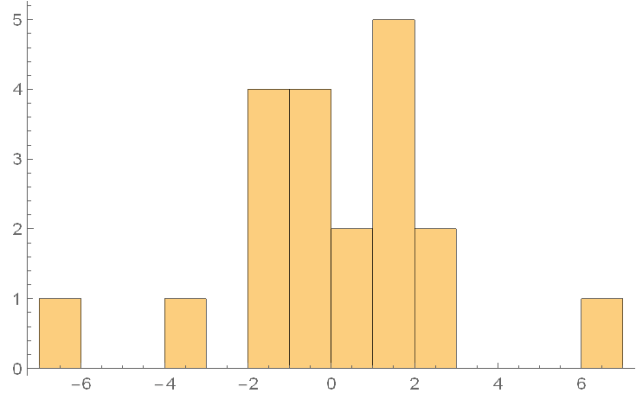}}
\caption{\label{050-16t_264t_2B} Distribution of average biases in the first 1000 primes for family $a_1=0$, $a_2=5$, $a_3=0$, $a_4=-16t^2$, $a_6=64t^2$.}
\end{center}
\end{figure}

As shown from the data, for most families, it is inconclusive whether the 6th moments have negative biases. For three of the five families, $50\%$ of the groups of primes have negative biases, which strongly suggests that the $p^{7/2}$ term averages to $0$, and the $p^3$ term is drowned out by the fluctuations of $p^{7/2}$ term.

For the family $a_1=1$, $a_2=t$, $a_3=-1$, $a_4=-t-1$, $a_6=0$ and family $a_1=1$, $a_2=1$, $a_3=-1$, $a_4=t$, $a_6=0$, it is likely that the $p^{7/2}$ term of their 6th moment sums  have negative biases(around $-0.5$ and $-0.6$ by Figure \ref{4_6moments}).

For the rank 6 family $a_1=0$, $a_2=2(16660111104 t) + 811365140824616222208$, $a_3=0$, $a_4=[2(-1603174809600)t-26497490347321493520384](t^2+2t-8916100448256000000+1)$, $a_6=[2(2149908480000)t+343107594345448813363200](t^2+2t-8916100448256000000+1)^2$, the average bias of the sixth moments sums for the first $1000$ primes is $2.26$. Figure \ref{811365140824616222208_6th} is a histogram plot of the distribution of the average biases among the $100$ groups of $10$ primes. $69$ of the $100$ groups of primes have positive biases, which suggests that it is likely that the sixth moment of this family has a positive $p^{7/2}$ term. 

\begin{figure}[ht]
\begin{center}
\scalebox{0.9}{\includegraphics{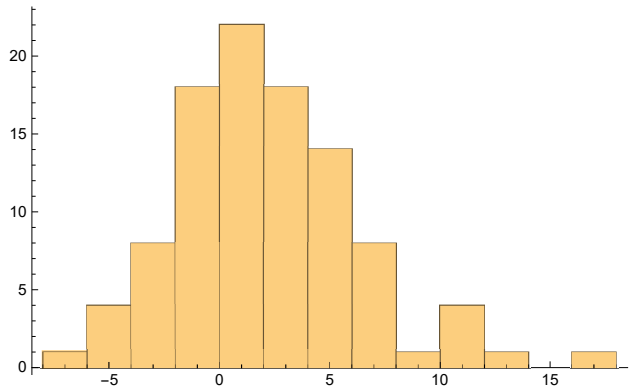}}
\caption{\label{811365140824616222208_6th} Distribution of average biases in the first 1000 primes for a rank 6 family.}
\end{center}
\end{figure}

To sum up, while we are not able to tell if the negative bias exists in the higher even moments, the data is at least consistent with the first lower order term averaging to zero or negative for families with smaller ranks. Thus our numerics support a weaker form of the bias conjecture: the first lower order term does not have a positive bias for smaller rank families. For families with ${\rm rank}(E(\Q))\ge4$, it is likely that the negative bias conjecture does not hold.


\section{Biases in the Third, Fifth, and Seventh Moments}
We now explore the third, fifth, and seventh moments of the Dirichlet coefficients of  elliptic curve $L$-functions. By the Philosophy of Square-Root Cancellation, $p$ times the third moment, $p$ times the fifth moment, and $p$ times the seventh moment  should have size $p^2$, $p^3$, $p^4$ respectively (we are multiplying by $p$ to remove the $1/p$ averaging). For example, the third moment is a sum of $p$ terms, each of size $\sqrt{p}^3$. Thus as these are signed quantities, we expect the size to be on the order of $\sqrt{p} \cdot p^{3/2}$ .

We believe that there are bounded functions $c_{E,3}(p)$, $c_{5,E}(p)$, and $c_{7,E}(p)$ such that
\be pA_{3,E}(p) \ = \ c_{3,E}(p) p^2+O(p^{3/2}), \ \ \ A_{5,E}(p) \ = \ c_{5,E}(p) p^3+O(p^{5/2}), \ \ \ A_{7,E}(p) \ = \ c_{7,E}(p) p^4+O(p^{7/2});\ee our data supports these conjectures. Unlike the second, fourth, and sixth moments, the coefficient of the leading term can vary with the prime in the third, fifth, and seventh moments. We calculated the average values of $c_{3,E}(p)$, $c_{5,E}(p)$, and $c_{7,E}(p)$ for each elliptic curve family by dividing the size of the main term ($p^2$ for third moment, $p^3$ for fifth moment, and $p^4$ for the seventh moment); see Figure \ref{3_5moments}.

\begin{figure}[ht]
\begin{center}
\scalebox{1.2}{\includegraphics{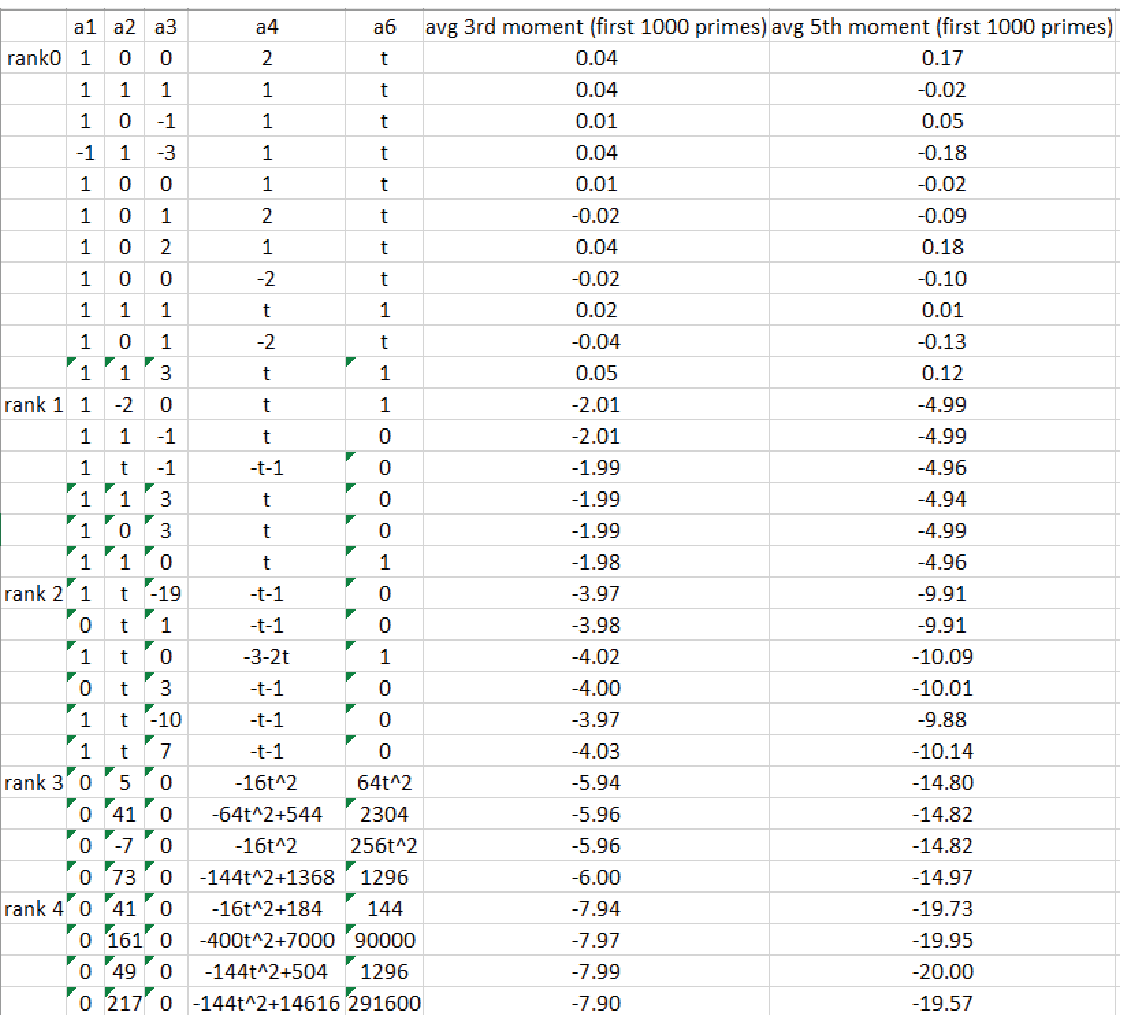}}
\caption{\label{3_5moments} Numerical data for the average constant for the main term of 3rd and 5th moments sums.}
\end{center}
\end{figure}

Our data suggests an interesting relationship between the average constant value for the main term and the rank of elliptic families for these odd moments.

\begin{conjecture}
The average value of the main term of the 3rd moment is $-2 {\rm rank}(E(\Q)) p^2$. 
\end{conjecture}

\begin{conjecture}
The average value of the main term of the 5th moment is $-5 {\rm rank}(E(\Q)) p^3$. 
\end{conjecture}

\begin{figure}[ht]
\begin{center}
\scalebox{1.2}{\includegraphics{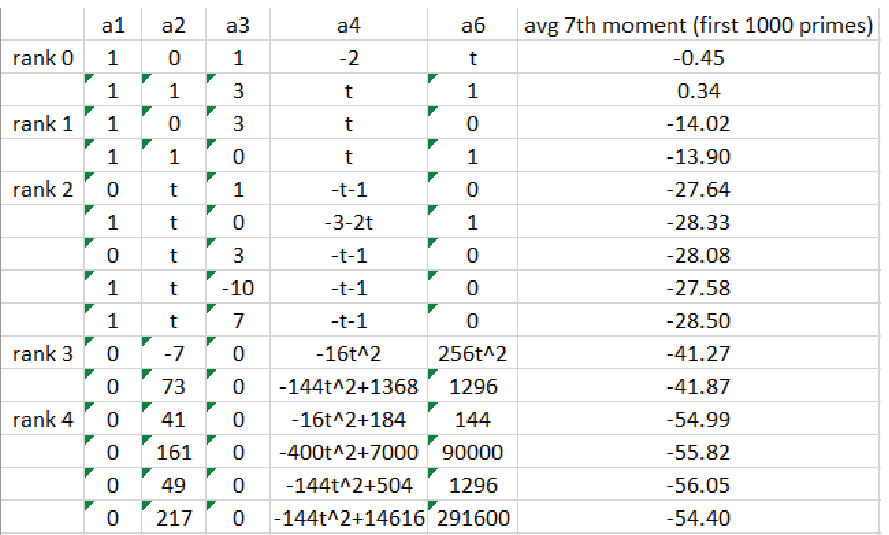}}
\caption{\label{7moments} Numerical data for the average constant for the main term of 7th moments sums.}
\end{center}
\end{figure}

\begin{conjecture}
The average value of the main term of the 7th moment is $-14 {\rm rank}(E(\Q)) p^4$.
\end{conjecture}

\begin{conjecture}
Let $C_n$ be the n-th term of the Catalan numbers. For $k \in \mathbf{Z}^+$, the average value of the main term of the $2k+1$ th moment is $-C_{k+1} {\rm rank}(E(\Q)) p^{k+1}$.
\end{conjecture}

We can try to analyze the third, fifth and seventh moments the same way as we did the fourth and sixth. In doing so, we would obtain expansions that do have terms related to the first moment (and hence by the Rosen-Silverman theorem the rank of the group of rational solutions); unfortunately there are other terms that arise now, due to the odd degree, that are not present in the even moments and which we cannot control as easily. We thus leave a further study of these odd moments as a future project.

\section{Future work}

Natural future questions are to continue investigating the second moment bias conjecture in more and more families, theoretically if possible, numerically otherwise. Since the bias in the second moments doesn't imply biases in higher moments, we can also explore whether there is a corresponding negative bias conjecture for the higher even moments. As these will involve quartic or higher in $t$ Legendre sums, it is unlikely that we will be able to  compute these in closed form, and thus will have to resort to analyzing data, or a new approach through algebraic geometry and cohomology theory (Michel proved that the lower order terms are related to cohomological quantities associated to the elliptic curve). 

Any numerical exploration will unfortunately be quite difficult in general, as there is often a term of size $p^{3/2}$ which we believe averages to zero for some families, but as it is $\sqrt{p}$ larger than the next lower order term, it completely drowns out that term and makes it hard to see the bias.

For the odd moments, our numerical explorations suggest that the bias in the first moment, which is responsible for the rank of the elliptic curve over $\Q(T)$, persists. A natural future project is to try to extend Michel's work to prove our conjectured main term formulas for the odd moments.

\section{Acknowledgement}
I would like to thank my mentor, Professor Steven J. Miller, for guiding me throughout the research process. Without his guidance, I would not have been able to learn the material and complete the research in this short period of time. 

I am also grateful to my parents and friends for their unwavering support. 

\section{Declaration of academic honesty}
I hereby confirm that the paper is the result of my own independent scholarly work under the guidance of the instructor, and that in all cases material from the work of others (in books, articles, essays, dissertations, and on the internet) is acknowledged, and quotations and paraphrases are clearly indicated. No material other than that listed has been used.

\appendix
\section{Linear and Quadratic Legendre Sums}\label{sec:linquadsums}
\setcounter{equation}{0}

The Dirichlet coefficients of elliptic curve $L$-functions can be written as cubic Legendre sums; while we do not have closed form expressions for these in general, we do for linear and quadratic sums. The proofs below are standard computations; see for example \cite{BEW, Mi1}. We include them for completeness.

\subsection{Linear Legendre Sums}

\begin{lemma}\label{lablegsumlinear} We have
\begin{equation}
S(n)\ :=\ \zsum{x} \js{ax+b} \ = \ \Bigg\{ {\ p\cdot \js{b} \ \
\mbox{{\rm if}} \ p \ | \ a \atop 0 \ \mbox{{\rm otherwise.}}}
\end{equation}
\end{lemma}

\begin{proof} When $p\mid a$, we have
\begin{eqnarray}
S(n) &\ =\ & \zsum{x} \js{b} \ = \  p\cdot \js{b}. 
\end{eqnarray}
When $p\nmid a$, $\gcd(p,a)=1$ and we can send $x$ to $a^{-1}x$, which yields
\begin{eqnarray}
S(n) &\ = \ & \zsum{x} \js{x+b}  \ = \  \zsum{x} \js{x} \ = \ 0.
\end{eqnarray}
\end{proof}

\subsection{Factorizable Quadratics in Sums of Legendre Symbols}

Before analyzing the most general quadratic Legendre sum, we first do an important special case. The following proof is directly copied from \cite{Mi1}, and is included for the convenience of the reader.

\begin{lemma}\label{lablegsumsimplequadratic} For $p > 2$
\begin{equation}
S(n) = \zsum{x} \js{n_1 + x} \js{n_2 + x} = \Bigg\{ {\ p-1 \ \
\mbox{{\rm if}} \ p \ | \ n_1 - n_2 \atop -1 \ \mbox{{\rm otherwise.}}}
\end{equation}
\end{lemma}

\begin{proof} Shifting $x$ by $-n_2$, we need only prove the lemma when
$n_2 = 0$. Assume $(n,p) = 1$ as otherwise the result is trivial.
For $(a,p) = 1$ we have
\begin{eqnarray}
S(n) &\ =\ & \sum_{x=0}^{p-1} \js{n + x} \js{x} \nonumber\\
     & = & \zsum{x} \js{n + a^{-1} x} \js{a^{-1} x} \nonumber\\
     & = & \zsum{x} \js{an + x} \js{x} = S(an).\end{eqnarray}

Hence
\begin{eqnarray}
S(n) &\ =\ & \frac{1}{p-1} \osum{a} \zsum{x} \js{an+x} \js{x} \nonumber\\
     & = & \frac{1}{p-1} \zsum{a} \zsum{x} \js{an+x} \js{x} -
     \frac{1}{p-1} \zsum{x} \js{x}^2 \nonumber\\
     & = & \frac{1}{p-1} \zsum{x} \js{x} \zsum{a} \js{an+x} - 1     \nonumber\\
     & = & 0 - 1 = -1.
\end{eqnarray}
\end{proof}

Where do we use $p > 2$? We used $\sum_{a=0}^{p-1} \js{an+x} = 0$
for $(n,p) = 1$. This is true for all odd primes (as there are
$\frac{p-1}{2}$ quadratic residues, $\frac{p-1}{2}$ non-residues,
and $0$); for $p=2$, there is one quadratic residue, no
non-residues, and $0$. As we never need to use this lemma for $p =
2$, this complication will not
affect any of our proofs.

\subsection{General Quadratics in Sums of Legendre Symbols}

The following proof is directly copied from \cite{Mi1}, and is included for the convenience of the reader.

\begin{lemma}\label{labquadlegsum} Assume
$a$ and $b$ are not both zero mod $p$ and $p > 2$. Then
\begin{equation}
\zsum{t} \js{at^2 + bt + c} = \Bigg\{ {(p-1)\js{a} \ {\rm if} \ p \ | \
b^2 - 4ac \atop -\js{a} \ {\rm otherwise.}}
\end{equation}
\end{lemma}

\begin{proof} Assume $a \not\equiv 0 (p)$ as otherwise the proof is
trivial. Let $\delta = 4^{-1}(b^2 - 4ac)$. Then
\begin{eqnarray}
\zsum{t} \js{at^2 + bt + c} & \ =\ & \zsum{t} \js{a^{-1}} \js{a^2 t^2
+ bat + ac} \nonumber\\ & = & \zsum{t} \js{a} \js{t^2 + bt + ac}
\nonumber\\ & = & \zsum{t} \js{a} \js{t^2 + bt + 4^{-1}b^2 + ac -
4^{-1} b^2} \nonumber\\ & = & \zsum{t} \js{a} \js{ (t + 2^{-1}b)^2
- 4^{-1}(b^2-4ac)} \nonumber\\ & = & \zsum{t} \js{a} \js{t^2 -
\delta} \nonumber\\ & = & \js{a} \zsum{t} \js{t^2 -\delta}.
\end{eqnarray}

If $\delta \equiv 0 (p)$ we get $p-1$. If $\delta = \eta^2, \eta
\neq 0$, then by Lemma \ref{lablegsumsimplequadratic}
\begin{equation}
\zsum{t} \js{t^2 - \delta} = \zsum{t} \js{t-\eta} \js{t+\eta} =
-1.
\end{equation}

We note that $\zsum{t} \js{t^2 - \delta}$ is the same for all
non-square $\delta$'s (let $g$ be a generator of the
multiplicative group, $\delta = g^{2k+1}$, change variables by $t
\to g^k t$). Denote this sum by $S$, the set of non-zero
squares by $\mathcal{R}$, and the non-squares by $\mathcal{N}$.
Since $\zsum{\delta} \js{t^2 - \delta} = 0$ we have

\begin{eqnarray}
\zsum{\delta} \zsum{t} \js{t^2 - \delta} & = & \zsum{t} \js{t^2} +
\sum_{\delta \in \mathcal{R}} \zsum{t} \js{t^2 - \delta} +
\sum_{\delta \in \mathcal{N}} \zsum{t} \js{t^2 - \delta}
\nonumber\\ & = & (p-1) + \frac{p-1}{2}(-1) + \frac{p-1}{2}S = 0
\end{eqnarray}

Hence $S = -1$, proving the lemma. \end{proof}


\section{Forms of 4th and 6th moments sums}

\subsection{Tools for higher moments calculations}

The Dirichlet Coefficients of the elliptic curve $L$-function can be written as\be a_t(p)\ = \ \sqrt{p}(e^{i\theta_t(p)}+e^{-i\theta_t(p)})\ = \ 2\sqrt{p}\cos (\theta_t(p)), \ee with $\theta_t(p)$ real; this expansion exists by Hasse's theorem, which states $|a_t(p)| \le 2 \sqrt{p}$. Define
\be {\rm sym}_k(\theta)\ := \ \frac{\sin((k+1)\theta)}{\sin \theta}.\ee  
By the angle addition formula for sine,
\be {\rm sym}_k(\theta)\ = \ {\rm sym}_{k-1}(\theta)\cos \theta+ \cos(k\theta) \ee
When $k=1$, we have
\be {\rm sym}_1(\theta)\ = \ 2\cos \theta \ee
Michel \cite{Mic} proved that
\be \sum_{t(p)} {\rm sym}_k(\theta_t(p))\ = \ O(\sqrt{p}), \ee where the big-Oh constant depends only on the elliptic curve and $k$; thus while we should have a $k$ subscript in the implied constant, as $k$ is fixed in our investigations we omit it for notational simplicity.

\subsection{Form of 4th moments sums}\label{4th_moment_form}

A lot is known about the moments of the $a_t(p)$ for a fixed elliptic curve $E_t$. However, as we are only concerned with averages over one-parameter families, we do not need to appeal to any results towards the Sato-Tate distribution, and instead we can directly prove convergence of the moments on average to the moments of the semicircle. In particular, the average of the $2m$-th moments has main term $\frac1{m+1}\ncr{2m}{m} p^{m-1}$. The coefficients $\frac1{m+1}\ncr{2m}{m}$ are the Catalan numbers, and the first few main terms of the even moments are $p, 2p^2, 5p^3$ and $14p^4$.

\begin{lemma} The average fourth moment of an elliptic surface with $j(T)$ non-constant has main term $2p^2$: \be \sum_{t(p)}{a_t}^4(p)\ = \  2p^3+O(p^\frac{5}{2}).\ee
\end{lemma}

\begin{proof} We have to compute
\be {a_t}^4(p)\ = \ 16p^2\cos^4 \theta_t(p). \ee We first collect some useful trigonometry identities:
\bea \cos(2\theta)&\ = \ &2\cos^2(\theta)-1\nonumber\\
\cos^2(\theta)&\ = \ &\frac{1}{2}\cos(2\theta)+\frac{1}{2}.\eea

We use these to re-write $\cos^4\theta$ in terms of quantities we can compute:
\bea \cos^4(\theta)&\ = \ &\frac{1}{4}\cos^2(2\theta)+\frac{1}{2}\cos(2\theta)+\frac{1}{4} \nonumber\\
&\ = \ &\frac{1}{8}\cos(4\theta)+\frac{1}{2}\cos(2\theta)+\frac{3}{8}\nonumber\\
&\ = \ &\frac{1}{8}[{\rm sym}_4(\theta)-{\rm sym}_3(\theta)\cos \theta]+\frac{1}{2}\cos(2\theta)+\frac{3}{8}.
\eea

The following expression will arise in our expansion, so we analyze it first:
\bea -\frac{1}{8}{\rm sym}_3(\theta)\cos \theta&\ = \ &-\frac{1}{8}\frac{\sin(4\theta)}{\sin \theta}\cos \theta \nonumber\\
 &\ = \ &-\frac{1}{8}\frac{2\sin(2\theta)\cos(2\theta)}{\sin \theta}\cos \theta \nonumber\\
 &\ = \ &-\frac{1}{8}\frac{2\cdot 2\sin \theta \cos \theta \cos(2\theta)}{\sin \theta}\cos \theta \nonumber\\
 &\ = \ &-\frac{1}{2}\cos^2 \theta \cos(2\theta)\nonumber\\
 &\ = \ &-\frac{1}{2}\cos^2 \theta (2\cos^2 \theta - 1)\nonumber\\
 &\ = \ &-\cos^4 \theta+\frac{1}{2}\cos^2 \theta \nonumber\\
16p^2\cdot \left(-\frac{1}{8}{\rm sym}_3(\theta)\cos \theta\right)&\ = \ &-16p^2\cos^4 \theta+8p^2\cos^2 \theta \nonumber\\
&\ = \ & -16p^2\cos^4 \theta+2p\cdot {a_t}^2(p).
\eea

Thus
\bea
16p^2\cos^4 \theta&\ = \ & 2p^2{\rm sym}_4\theta -16p^2\cos^4 \theta +2p\cdot {a_t}^2(p)+4p\cdot {a_t}^2(p)-2p^2\nonumber\\
2\cdot(16p^2\cos^4 \theta)&\ = \ & 2p^2{\rm sym}_4\theta+ 6p\cdot {a_t}^2(p)-2p^2\nonumber\\
\sum_{t(p)}(16p^2\cos^4 \theta)&\ = \ & p^2\sum_{t_(p)}{\rm sym}_4\theta+ 3p\sum_{t_(p)} {a_t}^2(p)-p^3\nonumber\\
\sum_{t(p)}{a_t}^4(p)&\ = \ &p^2\cdot O(\sqrt{p})+3p(p^2+O(p^\frac{3}{2}))-p^3 \nonumber\\
&\ = \ &2p^3+O(p^\frac{5}{2}),
\eea
as claimed. \end{proof}

\subsection{Form of 6th moments sums}\label{6th_moment_form}

\begin{lemma} The average sixth moment of an elliptic surface with $j(T)$ non-constant has main term $5p^3$: \be \sum_{t(p)}{a_t}^6(p)\ = \  5p^4+O(p^{\frac{7}{2}}).\ee
\end{lemma}

\begin{proof} We have
\bea {a_t}^6(p) &\ = \ & 64p^3\cos^6 \theta_t(p) \nonumber\\
 \cos(3\theta)&\ = \ &4\cos^3 \theta-3\cos \theta \nonumber\\
\cos^3\theta&\ = \ &\frac{\cos(3\theta)+3\cos \theta}{4}.\eea

We first expand $\cos^6\theta$:
\bea
\cos^6\theta&\ = \ &\frac{\cos^2(3\theta)+9\cos^2\theta+6\cos\theta\cos(3\theta)}{16} \nonumber\\
&\ = \ &\frac{\frac{1}{2}\cos(6\theta)+\frac{1}{2}+9[\frac{1}{2}\cos(2\theta)+\frac{1}{2}]+6\cos\theta[4\cos^3\theta-3\cos\theta]}{16} \nonumber\\
&\ = \ &\frac{10+\cos(6\theta)+9\cos(2\theta)+48\cos^4\theta-36\cos^2\theta}{32} \nonumber\\
&\ = \ &\frac{10+\cos(6\theta)+9\cos(2\theta)+48[\frac{1}{8}\cos(4\theta)+\frac{1}{2}\cos(2\theta)+\frac{3}{8}]-36\cos^2\theta}{32} \nonumber\\
&\ = \ &\frac{10+\cos(6\theta)+9\cos(2\theta)+48[\frac{1}{8}\cos(4\theta)+\frac{1}{2}\cos(2\theta)+\frac{3}{8}]-18\cos(2\theta)-18}{32} \nonumber\\
&\ = \ &\frac{10+\cos(6\theta)+6\cos(4\theta)+15\cos(2\theta)}{32} \nonumber\\
&\ = \ &\frac{\cos(6\theta)}{32}+\frac{10+6\cos(4\theta)+15\cos(2\theta)}{32} \nonumber\\
&\ = \ &\frac{{\rm sym}_6(\theta)-{\rm sym}_5(\theta)\cos\theta}{32}+\frac{10+6\cos(4\theta)+15\cos(2\theta)}{32}.
\eea

Next we find a formula for the symmetric function that will appear:
\bea
-\frac{1}{32}{\rm sym}_5(\theta)\cos \theta &\ = \ & -\frac{1}{32} (\frac{\sin(6\theta)}{\sin \theta})\cos \theta\nonumber\\
&\ = \ & -\frac{1}{32}\cos \theta (\frac{3\sin(2\theta)-4\sin^3(2\theta)}{\sin \theta}) \nonumber\\
&\ = \ & -\frac{1}{32}\cos \theta (\frac{6\sin\theta\cos\theta-32\sin^3\theta\cos^3\theta}{\sin \theta})\nonumber\\
&\ = \ & -\frac{1}{32}\cos \theta [6\cos\theta-32\sin^2\theta\cos^3\theta] \nonumber\\
&\ = \ & -\frac{3}{16}\cos^2\theta +(1-\cos^2\theta)\cos^4\theta \nonumber\\
&\ = \ & -\frac{3}{16}\cos^2\theta +\cos^4\theta-\cos^6\theta \nonumber\\
64p^3(-\frac{1}{32}{\rm sym}_5(\theta)\cos \theta) &\ = \ &-12p^3\cos^2\theta+64p^3\cos^4\theta-64p^3\cos^6\theta \nonumber\\
&\ = \ &-64p^3\cos^6\theta+4p{a_t}^4(p)-3p^2{a_t}^2(p).
\eea
Thus
\bea
64p^3\cos^6\theta&\ = \ &2p^3{\rm sym}_6(\theta)-64p^3\cos^6\theta+4p{a_t}^4(p)-3p^2{a_t}^2(p)+12p^3\cos(4\theta)+30p^3\cos(2\theta)+20p^3\nonumber\\
64p^3\cos^6\theta&\ = \ &p^3{\rm sym}_6(\theta)+2p{a_t}^4(p)-\frac{3}{2}p^2{a_t}^2(p)+6p^3\cos(4\theta)+15p^3\cos(2\theta)+10p^3.
\eea

We can re-express some of the terms above in a more convenient form:
\bea
15p^3\cos(2\theta)&\ = \ &15p^3(2\cos^2\theta-1) \nonumber\\
&\ = \ &30p^3\cos^2\theta-15p^3 \nonumber\\
&\ = \ &\frac{15}{2}p^2{a_t}^2(p)-15p^3
\eea
and
\bea
6p^3\cos(4\theta)&\ = \ &6p^3[{\rm sym}_4(\theta)-{\rm sym}_3(\theta)\cos\theta] \nonumber\\
&\ = \ &6p^3[{\rm sym}_4(\theta)-8\cos^4\theta+4\cos^2\theta] \nonumber\\
&\ = \ &6p^3{\rm sym}_4(\theta)-3p{a_t}^4(p)+6p^2{a_t}^2(p).
\eea

Thus
\bea
64p^3\cos^6\theta & \ = \ & p^3{\rm sym}_6\left(\theta\right)+2p{a_t}^4\left(p\right)-\frac{3}{2}p^2{a_t}^2\left(p\right)+6p^3{\rm sym}_4\left(\theta\right) -3p{a_t}^4\left(p\right)+6p^2{a_t}^2\left(p\right)\nonumber\\ & & \ \ \ \ \ +\ \frac{15}{2}p^2{a_t}^2\left(p\right)-15p^3+10p^3 \nonumber\\
\sum_{t\left(p\right)}64p^3\cos^6\theta &\ = \  & \sum_{t\left(p\right)}[p^3{\rm sym}_6\left(\theta\right)+2p{a_t}^4\left(p\right)-\frac{3}{2}p^2{a_t}^2\left(p\right)+6p^3{\rm sym}_4\left(\theta\right)-3p{a_t}^4\left(p\right)+6p^2{a_t}^2\left(p\right)\nonumber\\ & & \ \ \ \ \ +\ \frac{15}{2}p^2{a_t}^2\left(p\right)-15p^3+10p^3].
\eea

Therefore
\bea
\sum_{t(p)}{a_t}^6(p)&\ = \ &p^3\sum_{t(p)}{\rm sym}_6(\theta)+6p^3\sum_{t(p)}{\rm sym}_4(\theta)-p\sum_{t(p)}{a_t}^4(p)+12p^2\sum_{t(p)}{a_t}^2(p)-5p^4 \nonumber\\
&\ = \ &p^3O(\sqrt{p})+6p^3O(\sqrt{p})-p(2p^3+O(p^{\frac{5}{2}}))+12p^2(p^2+O(p^{\frac{3}{2}}))-5p^4 \nonumber\\
&\ = \ &5p^4+O(p^{\frac{7}{2}}),
\eea completing the proof. \end{proof}


\section{Mathematical code}

\subsection{Finding formulas for second moment sums}
Below is an example of mathematica code to find the formulas for the elliptic curve family $y^2+xy+y=x^3+x^2+x+t$. It would take four to five days to run 1000 primes for a family on an old laptop which was devoted to this problem. To use one chooses the values of the input polynomials $a_1, a_2, a_3, a_4, a_6$ and the file name (which for convenience we take to be related to the coefficients).

\begin{lstlisting}
a1[t_]:= 1;
a2[t_]:= 1;
a3[t_]:= 1;
a4[t_]:= 1;
a6[t_]:= t;

sums[a1_, a2_, a3_, a4_, a6_, pstart_, pend_, powerof2_, powerof3_,
   familyname_] := Module[{},
   (* defines the polynomials we will use *)
   (* num moments is the number of moments we do *)
   (* pstart and pend is index of first and last prime studied *)
   (* we will find the formula for the second moment sum according to
the congruence classes of 2^powerof2*3^powerof3 *)
   b2[t_]:= a1[t]^2 + 4*a2[t];
   b4[t_]:= 2*a4[t] + a1[t]*a3[t];
   b6[t_]:= a3[t]^2 + 4*a6[t];
   c4[t_]:= b2[t]^2 - 24*b4[t];
   c6[t_]:= -b2[t]^3 + 36*b2[t]*b4[t] - 216*b6[t];

   (*primeton[n]=m, the mth prime is n*)
   numdo = 1000;
   For[n = 1, n <= Prime[numdo], n++, primeton[n] = 0];
   For[n = 1, n <= numdo, n++, primeton[Prime[n]] = n];

   (*creates lists to put different primes into different prime
congruence class*)

   For[m = 1, m <= 2^powerof2*3^powerof3, m++,
    {
     primeModgroup[m] = {};
     }];

   For[m = 3, m <= pend, m++,
    {
     p = Prime[m];
     j = Mod[p, 2^powerof2*3^powerof3];
     primeModgroup[j] = AppendTo[primeModgroup[j], p];
     }];

   (* initializes moment lists to empty *)
   moment = {};
   (* loops over the primes we study *)
   For[k = pstart, k <= pend, k++,
    {
     p = Prime[k];
     momenttemp = 0;
     For[t = 0, t <= p-1, t++,
      {
       aept = Sum[JacobiSymbol[x^3 - 27*c4[t]*x - 54*c6[t], p],
       {x, 0, p-1}];
       currentvalue = 1;
       currentvalue = currentvalue * aept*aept;
       momenttemp = momenttemp + currentvalue;
       }]; (* end of t loop to compute a_E(t)*)
     moment = AppendTo[moment, {p, momenttemp}];
     Print["Working with prime p = ", p, "."];
      Print["The second moment sum is ", moment[[k - pstart + 1, 2]]];
     Print[" "];
     }];
   savename = ToString[familyname];
   Print["We are saving the file to ", familyname];
   SetDirectory[" "];
   Put[savelist, savename];
   Put[{ moment[2]}, savename];
   (*Saving the data of second moment sums on the computer*)

   For[m = 1, m <= 2^powerof2*3^powerof3, m++,
    {
     If[Length[primeModgroup[m]] != 0, {
        If[Length[primeModgroup[m]] < 4,
         Print["prime mod" , 2^powerof2*3^powerof3, ":", m,
           "There are not enough primes"] && Continue[]];

        primeList = primeModgroup[m];
        For[i = 2, i <= Length[primeList], i++,
         {
          prime[i] = primeList[[i]];(*find the i-th prime*)
          nthPrime[i] = primeton[prime[i]];

          sum[i] = moment[[nthPrime[i] - pstart + 1, 2]]- (prime[i])^2;
          (*second moment sum=p^2+ap+b,sum[i]=ap+b*)
          }];
        a = (sum[3] - sum[2])/(prime[3] - prime[2]);
        b = sum[2] - a*prime[2];

        polywork = 1;
        For[i = 4, i <= Length[primeList], i++,
         {
          If[sum[i] != a*prime[i] + b,
            {
             polywork = 0;
             i = Length[primelist] + 100;
             }];
          }];(*end of if statement*)(*If polywork does not work for
          the i-th prime,set polywork to 0*)

        If[polywork == 1,
         Print["prime mod", 2^powerof2*3^powerof3, ":", m,
          "second moment sum=p^2 + (",a,") p + ",b],
         Print["prime mod" , 2^powerof2*3^powerof3, ":", m,
          "We can't find the formula."]];
        }];
     }];
   ]; (*end of module*)

sums[a1, a2, a3, a4, a6, 1, 440, 3, 3 , "1111t.dat"]
\end{lstlisting}

\end{document}